\documentclass[11pt]{amsart}
\usepackage{amssymb}
\usepackage{graphicx}
\usepackage{amscd}
\usepackage{amsmath}
\usepackage{amsfonts,latexsym,amstext}

\usepackage [T1]{fontenc}
\usepackage [latin1]{inputenc}
\usepackage{color}
\usepackage[all]{xy}
\newtheorem*{theorem HaaMu PS}{Theorem}

\newtheorem{theorem}{Theorem}[section]
\newtheorem{definition}[theorem]{Definition}

\newtheorem{lemma}[theorem]{Lemma}
\newtheorem{proposition}[theorem]{Proposition}

\theoremstyle{remark}
\newtheorem{example}[theorem]{Example}
\newtheorem{remark}[theorem]{Remark}

\begin{document}

\title{Bilinear ideals in operator spaces}

\author{Ver\'{o}nica Dimant}
\author{Maite Fern\'{a}ndez-Unzueta}

\thanks{The first author was partially supported by CONICET PIP 0624 and PAI-UDESA 2011. The second author was partially supported by CONACYT 182296}

\address{Departamento de Matem\'{a}tica, Universidad de San
Andr\'{e}s, Vito Dumas 284, (B1644BID) Victoria, Buenos Aires,
Argentina and CONICET.} \email{vero@udesa.edu.ar}

\address{Centro de Investigaci\'{o}n en Matem\'{a}ticas (Cimat), A.P. 402 Guanajuato, Gto., M\'{e}xico} \email{maite@cimat.mx}


\begin{abstract}
 We introduce a concept  of bilinear ideal of jointly completely bounded mappings between operator spaces. In particular, we study the bilinear ideals  $\mathcal{N}$ of completely nuclear,  $\mathcal{I }$ of  completely integral,  $\mathcal{E}$  of completely extendible  bilinear mappings,  $\mathcal{MB}$ multiplicatively bounded and its symmetrization $\mathcal{SMB}$.  We  prove some basic properties of them, one of which is the fact that
  $\mathcal{I}$ is naturally identified with the ideal of (linear)  completely integral mappings on the  injective operator space tensor product.

\end{abstract}

\maketitle

{\section{Introduction and Preliminaries}

Let  $V, W$ and $X$ be operator spaces. If   we consider  the underlying vector space structure,
the relations
\begin{equation}\label{linear isomorphisms}
  Bil(V\times W, X)\;  \stackrel{\nu}{\simeq} \;  \mathcal L(V\otimes W ,X) \;\stackrel{\rho}{\simeq}\; \mathcal L(V,\mathcal L(W, X)) \;
 \end{equation}
hold through the two natural linear isomorphisms $\nu$, $\rho$.
In order for $\nu$ and $\rho$  to induce natural morphisms in the  operator space category, it is necessary to have appropriately defined   an operator space tensor norm  on $V\otimes W$ and  specific classes of linear and bilinear mappings. This is the case, for instance,  of  the so called projective operator space tensor norm   $ \|\cdot\|_{\wedge{}}$, the completely bounded maps    and the  jointly completely bounded bilinear mappings, where   $\nu$ and  $\rho$ induce the following completely bounded isometric isomorphisms:
$$
  \mathcal{JCB}(V\times W, X) \;  {\simeq} \;  \mathcal{CB}(V\widehat\otimes W ,X) \; {\simeq}\; \mathcal{CB}(V,\mathcal{CB}(W, X)).
 $$

 There are many possible ways to
provide $V\otimes W$  with an operator space tensor norm and, of course, to define  classes of mappings. Several authors, inspired by the  success  that the  study of the relations between tensor products and mappings    has had  in the Banach space setting, have systematically     study  some analogous  relations for operator spaces.  This is the case, for instance, of the completely nuclear and completely integral linear mappings (see \cite[Section III]{ER-libro}).

In  this paper we follow this   approach as well, but with  the attention focused   on the relations involving $\nu$,  the  isomorphism in (\ref{linear isomorphisms}) which  concerns bilinear mappings.  In Section
\ref{sect: bilinear ideals}  we   introduce the notion of  an ideal of completely bounded bilinear mappings and    study its general properties.  In Section \ref{sect: compl nuclear compl integral} we define the ideals of completely nuclear and completely integral bilinear mappings. The main result proved  here is that the ideal of  completely integral bilinear mappings is naturally identified with the ideal of completely integral linear mappings on the  injective operator space tensor product, that is $
\mathcal{I}(V\times W, X)\cong \mathcal{L_{\mathcal{I}}}(V\overset\vee\otimes W,X)
$ (see Theorem \ref{thm: main for integrals}).  This implies that, contrary to the result for Banach spaces, the relation $\mathcal{I}(V\times W)\cong \mathcal{L_{\mathcal{I}}}(V, W^*)$ does not always hold. Indeed, it holds if and only if $W$ is locally reflexive.

  The ideal  $\mathcal{E}$   of bilinear completely extendible mappings is   introduced  in  Section \ref{sect: compl extendibles}.   We prove in
Proposition \ref{prop: extendible-tensor} that $\mathcal{E}$ gives rise, through duality,  to an operator space tensor product $\eta$ such that  $
\left(V\overset\eta\otimes W\right)^*\cong \mathcal E(V\times W)$. In Section \ref{sect: SMB ideal}  we consider  the ideal  $\mathcal{SMB}$ of
 symmetrized multiplicatively bounded mappings, which is the symmetrization of the ideal $\mathcal{MB}$  of multiplicatively bounded mappings.
The following theorem summarizes the inclusion relations among all these  bilinear ideals:

 \begin{theorem}\label{thm: inclusion relations} Let $V, W$ and $X$ be operator spaces. Then, we have the following complete contractive inclusions:
 \begin{enumerate}
 \item[(a)]  $
 \mathcal{N}(V\times W, X)\subset \mathcal{I}(V\times W, X)\subset \mathcal{MB}(V\times W, X)\subset \mathcal{SMB}(V\times W, X)\subset \mathcal{JCB}(V\times W, X).
$
  \item[(b)]   $ \mathcal{I}(V\times W, X)\subset \mathcal{E}(V\times W, X)\subset \mathcal{JCB}(V\times W, X). $
 \item[(c)]  $ \mathcal{MB}(V\times W, \mathcal{L}(H))\subset\mathcal{SMB}(V\times W, \mathcal{L}(H))\subset \mathcal{E}(V\times W, \mathcal{L}(H)))\subset \mathcal{JCB}(V\times W, \mathcal{L}(H)). $
  \end{enumerate}
 \end{theorem}

 In  Section  \ref{sect: examples} we    prove the inclusions  and      provide examples  to  distinguish the ideals.

\

We now recall some basic concepts about operator spaces, mainly with respect to bilinear operators and tensor products. For a more complete presentation of these topics, see \cite{BleLeM, ER-libro, Pisier-libro}.
All vector spaces considered are over the complex numbers.  For a linear space $V$, we let  $M_{n\times m}(V)$  denote the set of all the $n\times m$ matrices of elements in $V$.  In the case  $n=m$, the notation  is simplified to set  $M_{n\times n}(V)=M_n(V)$. If $V$ is the scalar field we just write $M_{n\times m}$ and $M_n$, respectively.   For $\alpha\in M_{n\times m}$, its norm  $\|\alpha\|$ will be considered  as an operator from $\ell_2^m$ to $\ell_2^n$.

 Given $v=(v_{i,j})\in M_n(V)$ and $w=(w_{k,l})\in M_m(V)$,  $v\oplus w\in M_{n+m}(V)$ stands for the matrix
$$
v\oplus w=\left(
            \begin{array}{cc}
              (v_{i,j}) & 0 \\
              0 & (w_{k,l}) \\
            \end{array}
          \right).
$$

A matrix norm $\|\cdot\|$ on a linear space $V$ is an assignment of a norm $\|\cdot\|_n$ on $M_n(V)$, for each $n\in \mathbb{N}$. A linear space $V$ is an
 \textbf{operator space} if it  is  endowed with a matrix norm  satisfying:
\begin{enumerate}
\item[\textbf{M1}] $\|v\oplus w\|_{n+m}=\max\{\|v\|_n,\|w\|_m\}$, for all $v\in M_n(V)$ and $w\in M_m(V)$.
\item[\textbf{M2}] $\|\alpha v \beta\|_m\leq \|\alpha\|\cdot \|v\|_n \cdot \|\beta\|$, for all $v\in M_n(V)$, $\alpha\in M_{m\times n}$ and $\beta\in M_{n\times m}$.
\end{enumerate}
 We usually omit the subindex $n$ in the matrix norms and simply denote $\|\cdot\|$ instead of  $\|\cdot\|_n$. The inclusion $M_{n\times m}(V)\hookrightarrow M_{\max\{n,m\}}(V)$ naturally endows the rectangular matrices with a norm.
Throughout the article, $V$, $W$, $X$, $Y$, $Z$, $U_1$, $U_2$ will denote operator spaces where the underlying normed space is complete (i.e. it is a Banach space).

Every  linear mapping $\varphi:V\to W$ induces, for each $n\in\mathbb{N}$,  a linear mapping  $\varphi_n: M_n(V)\to M_n(W)$ given by
$$
\varphi_n(v)=\left(\varphi(v_{i,j})\right), \textrm{ for all }v=(v_{i,j})\in M_n(V).
$$

 It holds that $\|\varphi\|=\|\varphi_1\|\leq \|\varphi_2\|\leq\|\varphi_3\|\leq ...$.  The \textbf{completely bounded norm} of $\varphi$ is defined by
$$
\|\varphi\|_{cb}=\sup_{n\in\mathbb{N}}\|\varphi_n\|.
$$

We say that $\varphi$ is \textbf{completely bounded} if $\|\varphi\|_{cb}$ is finite, that $\varphi$ is \textbf{completely contractive} if $\|\varphi\|_{cb}\le 1$ and that $\varphi$ is a \textbf{complete isometry} if each $\varphi_n:M_n(V)\to M_n(W)$ is an isometry. It is easy to see that $\|\cdot\|_{cb}$ defines a norm on the space $\mathcal{CB}(V,W)$ of all  completely bounded linear mappings from $V$ to $W$. The natural identification $M_n\left(\mathcal{CB}(V,W)\right)\cong \mathcal{CB}\left(V,M_n(W)\right)$ provides $\mathcal{CB}(V,W)$ with the structure of an operator space. Also, since $V^*=\mathcal{CB}(V,\mathbb{C})$, the dual of an operator space is again an operator space.

   In contrast to the linear case, a bilinear mapping  $\phi:V\times W\to X$ naturally induces not one, but two different bilinear mappings in the matrix levels. Some authors (see, for instance \cite{ER-libro, Sch}) use the name ``complete boundedness'' for the first  notion  and ``multiplicative boundedness'' or ``matrix complete boundedness'' for the second one, while others \cite{BleLeM, BlePau, Wi}  use  the name ``jointly complete boundedness'' for the first concept and ``complete boundedness'' for the second one. In order to avoid confusion, we will not  use the name  ``complete boundedness'' for bilinear mappings.

  So, given a bilinear mapping $\phi:V\times W\to X$, consider the  associated bilinear mapping  $\phi_n:M_n(V)\times M_n(W)\to M_{n^2}(X)$ defined, for each $n\in \mathbb{N}$, as follows:
$$
\phi_n(v,w)=\left(\phi(v_{i,j},w_{k,l})\right), \textrm{ for all }v=(v_{i,j})\in M_n(V), w=(w_{k,l})\in M_n(W).
$$

When  their norms are uniformly bounded, that is, when
$$
\|\phi\|_{jcb}\equiv\sup_{n\in\mathbb{N}}\|\phi_n\|<\infty,
$$
we say that $\phi$ is \textbf{jointly completely bounded}.
 It is plain to see that $\|\cdot\|_{jcb}$ is a norm on the space $\mathcal{JCB}(V\times W,X)$ of all  jointly completely bounded bilinear mappings from $V\times W$ to $X$. As in the linear setting, the identification
$$
M_n\left(\mathcal{JCB}(V\times W,X)\right)\cong \mathcal{JCB}\left(V\times W,M_n(X)\right).
$$
provides  $\mathcal{JCB}(V\times W,X)$  with an   operator space  structure.

 The second way to naturally associate  $\phi$ with a bilinear mapping  $\phi_{(n)}:M_n(V)\times M_n(W)\to M_n(X)$, for each $n\in\mathbb{N}$, involves the matrix product and it
  is given by
$$
\phi_{(n)}(v,w)=\left(\sum_{k=1}^n\phi(v_{i,k},w_{k,l})\right), \textrm{ for all }v=(v_{i,j})\in M_n(V), w=(w_{k,l})\in M_n(W).
$$

We say that $\phi$ is \textbf{multiplicatively bounded} if
$$
\|\phi\|_{mb}=\sup_{n\in\mathbb{N}}\|\phi_{(n)}\|<\infty.
$$

Again, it is easily seen that $\|\cdot\|_{mb}$ is a norm on the space $\mathcal{MB}(V\times W,X)$ of  all  multiplicatively bounded bilinear mappings from $V\times W$ to $X$. The identification
$$
M_n\left(\mathcal{MB}(V\times W,X)\right)\cong \mathcal{MB}\left(V\times W,M_n(X)\right)
$$
endows $\mathcal{MB}(V\times W,X)$ with matrix norms  that give  the structure of an operator space.

   We finish this section recalling  three  basic examples  from the theory of tensor products of operator spaces
    (the general notion is   in Definition  \ref{def: operator space tensor norm}):
   the  operator space projective tensor norm,  the operator space injective tensor norm and the operator space Haagerup tensor norm.

 Consider  two operator spaces $V$ and $W$. The definition  of  the first  norm uses the fact that
each element $u\in M_n(V\otimes W)$ can be written as:
\begin{equation}\label{escritura-proyectiva}
u=\alpha(v\otimes w)\beta
\end{equation}
 with $v\in M_p(V), w\in M_q(W), \alpha\in M_{n\times p\cdot q}, \beta\in M_{p\cdot q\times n}$,  for certain $p,q\in\mathbb{N}$, where $v\otimes w$ is the $p\cdot q\times p\cdot q$-matrix given by
\begin{eqnarray} \label{matrizota}
v\otimes w=\left(
            \begin{array}{cccccccc}
              v_{1,1}\otimes w_{1,1} & \cdots & v_{1,1}\otimes w_{1,q} & \cdots & \cdots & v_{1,p}\otimes w_{1,1} & \cdots & v_{1,p}\otimes w_{1,q} \\
              \vdots & \vdots & \vdots & \cdots & \cdots & \vdots & \vdots & \vdots \\
              v_{1,1}\otimes w_{q,1} & \cdots & v_{1,1}\otimes w_{q,q} & \cdots & \cdots & v_{1,p}\otimes w_{q,1} & \cdots& v_{1,p}\otimes w_{q,q} \\
              \cdots & \cdots & \cdots & \cdots & \cdots &\cdots & \cdots & \cdots \\
              \cdots & \cdots &\cdots & \cdots & \cdots & \cdots & \cdots & \cdots \\
               v_{p,1}\otimes w_{1,1} & \cdots & v_{p,1}\otimes w_{1,q} & \cdots & \cdots & v_{p,p}\otimes w_{1,1} & \cdots & v_{p,p}\otimes w_{1,q} \\
              \vdots & \vdots & \vdots & \cdots & \cdots & \vdots & \vdots & \vdots \\
              v_{p,1}\otimes w_{q,1} & \cdots & v_{p,1}\otimes w_{q,q} & \cdots & \cdots & v_{p,p}\otimes w_{q,1} & \cdots& v_{p,p}\otimes w_{q,q} \\
            \end{array}
          \right)
\end{eqnarray}

The \textbf{operator space projective tensor norm} of  $u\in M_n(V\otimes W)$ is defined as
$$
\|u\|_\wedge =\inf\{\|\alpha\|\cdot \|v\|\cdot\|w\|\cdot\|\beta\|: \textrm{ all representations of } u \textrm{ as in (\ref{escritura-proyectiva})}\}.
$$

 The \textbf{operator space injective tensor norm} of  $u\in M_n(V\otimes W)$ is defined as
$$
\|u\|_\vee =\sup\left\{\|(f\otimes g)_n(u)\|:\ f\in M_p(V^*), g\in M_q(W^*), \|f\|\le 1, \|g\|\le 1\right\}.
$$

The \textbf{operator space projective tensor product} $V\widehat\otimes W$
 and the \textbf{operator space injective tensor product} $V\overset\vee\otimes W$  are  the completion of $\left(V\otimes W, \|\cdot\|_\wedge\right)$ and the completion of $\left(V\otimes W, \|\cdot\|_\vee\right)$, respectively.

There is a natural completely isometric identification:
$$
\mathcal{JCB}(V\times W,X)\cong \mathcal{CB}(V\widehat\otimes W,X)\cong \mathcal{CB}(V,\mathcal{CB}(W,X)).
$$
 So,  in particular:
$$
\mathcal{JCB}(V\times W)\cong (V\widehat\otimes W)^*\cong \mathcal{CB}(V,W^*).
$$

The identification of  $ (V\overset\vee\otimes W)^*$ with a subset of bilinear mappings is done later, in Proposition \ref{prop: integrales=dual del inyectivo}.

Every $u\in M_n(V\otimes W)$ can be written as $u=v\odot w$, for certain matrices $v\in M_{n\times r}(V)$ and $w\in M_{r\times n}(W)$, where
 $$
v\odot w = \left(\sum_{k=1}^r v_{i,k}\otimes w_{k,j}\right).
$$
The \textbf{Haagerup tensor norm} is defined as:
$$
\|u\|_h =\inf\left\{\|v\|\cdot\|w\|:\ u=v\odot w,\ v\in M_{n\times r}(V), w\in M_{r\times n}(W), r\in\mathbb{N}\right\},
$$
while  the \textbf{Haagerup tensor product} $V\overset{h}\otimes W$ is the completion of $\left(V\otimes W, \|\cdot\|_h\right)$.

For any operator spaces $V$ and $W$, $\|\cdot\|_\vee$ and $\|\cdot\|_\wedge$ are, respectively,  the smallest and the largest operator space cross norms on $V\otimes W$.
In particular, for each $u\in M_n(V\otimes W)$ it holds that
$$
\|u\|_\vee\le \|u\|_h \le \|u\|_\wedge.
$$

The Haagerup tensor product is naturally associated with  multiplicatively bounded bilinear operators through the following identifcations:
$$
\mathcal{MB}(V\times W,X)\cong \mathcal{CB}(V\overset h\otimes W,X)\qquad\textrm{and}\qquad \mathcal{MB}(V\times W)\cong (V\overset h\otimes W)^*.
$$

\begin{remark}\label{extensiones}
We will use repeatedly along the text the following extension property for completely bounded linear mappings (see \cite[Theorem 4.1.5]{ER-libro}): if $V$ is a subspace of an operator space $W$ and $H$ is a Hilbert space, then every completely bounded linear map $\varphi:V\to \mathcal L(H)$ has a completely bounded extension $\overline\varphi:W\to \mathcal L(H)$ with $\|\varphi\|_{cb}=\|\overline\varphi\|_{cb}$.

Equivalently, this can be stated  as in  \cite[Theorem 1.6]{Pisier-libro}: if $V$, $W$ are operator spaces, $H$, $K$ are Hilbert spaces such that $V$ is a subspace of $\mathcal L(H)$ and $W$ is a subspace of $\mathcal L(K)$, then every completely bounded linear map $\varphi:V\to W$ has a completely bounded extension $\overline\varphi:\mathcal L(H)\to \mathcal L(K)$ with $\|\varphi\|_{cb}=\|\overline\varphi\|_{cb}$.
\end{remark}

\section{Bilinear ideals}\label{sect: bilinear ideals}

  The linear structure and the closedness by compositions are the basic properties required of  a subset of maps, in order to have  a suitable relation between  mappings spaces and tensor products.  These will be, precisely,  the defining properties of a {\sl bilinear ideal} (see Definition \ref{def: bilinear ideal}).  To deal with compositions, we need first to prove the following estimate:

\begin{lemma}
Let $\phi\in  M_n\left(\mathcal{JCB}(V\times W,X)\right)$, $r_1\in \mathcal{CB}(U_1,V)$, $r_2\in \mathcal{CB}(U_2,W)$, $s\in \mathcal{CB}(X,Y)$. Then $s_n\circ\phi\circ (r_1,r_2)$ is jointly completely bounded and
$$
\|s_n\circ\phi\circ (r_1,r_2)\|_{jcb}\le \|s\|_{cb}\cdot \|\phi\|_{jcb}\cdot \|r_1\|_{cb}\cdot \|r_2\|_{cb}.
$$
\end{lemma}

\begin{proof}
Let $\psi=s_n\circ\phi\circ (r_1,r_2)$.  It is easy to see that
$$
\psi_m= s_{n\cdot m^2}\circ \phi_m\circ ((r_1)_m, (r_2)_m).
$$
Thus, for every $m$,
$$
\|\psi_m\|\le \|s_{n\cdot m^2}\|\cdot \|\phi_m\|\cdot \|(r_1)_m\|\cdot \|(r_2)_m\|\le \|s\|_{cb}\cdot \|\phi\|_{jcb}\cdot \|r_1\|_{cb}\cdot \|r_2\|_{cb},
$$ and the conclusion follows.
\end{proof}

In accordance with the definition of an operator space ideal of linear mappings (see  \cite{EJR} and \cite{ER-libro}), we introduce:

\begin{definition}\label{def: bilinear ideal}
An \textbf{operator space bilinear ideal} $\mathfrak{A}$ is an assignment, to each group of three  operator spaces $V$, $W$ and $X$, of a linear  subspace $\mathfrak{A}(V\times W, X)$ of $\mathcal{JCB}(V\times W,X)$   containing all finite type continuous bilinear maps, together with an operator space matrix norm $\|\cdot\|_{\mathfrak{A}}$ such that:
\begin{enumerate}
\item[(a)] For all $\phi\in M_n(\mathfrak{A}(V\times W, X))$, $\|\phi\|_{jcb}\le \|\phi\|_{\mathfrak{A}}$.
\item[(b)] For all $\phi\in M_n(\mathfrak{A}(V\times W, X))$, $r_1\in \mathcal{CB}(U_1,V)$, $r_2\in \mathcal{CB}(U_2,V)$, $s\in \mathcal{CB}(X,Y)$, the matrix $s_n\circ\phi\circ (r_1,r_2)$ belongs to $M_n(\mathfrak{A}(U_1\times U_2, Y))$ and
$$
\|s_n\circ\phi\circ (r_1,r_2)\|_{\mathfrak{A}}\le \|s\|_{cb}\cdot \|\phi\|_{\mathfrak{A}}\cdot \|r_1\|_{cb}\cdot \|r_2\|_{cb}.
$$
\end{enumerate}
\end{definition}

  We now introduce the notion of tensor norm for operator spaces.

\begin{definition}\label{def: operator space tensor norm}
We say that $\alpha$ is an \textbf{operator space tensor norm} if $\alpha$ is an operator space matrix norm on each tensor product of operator spaces $V\otimes W$ that satisfies the following two conditions:
\begin{enumerate}
\item[(a)] $\alpha$ is a cross matrix norm,  that is, $\alpha(v\otimes w)=\|v\|\cdot\|w\|$, for all $v\in M_p(V)$, $w\in M_q(W)$, $p,q\in \mathbb{N}$.
\item[(b)] $\alpha$ fullfils the ``completely metric mapping property'': for every  $r_1\in \mathcal{CB}(U_1,V)$, $r_2\in \mathcal{CB}(U_2,W)$, the operator $r_1\otimes r_2: (U_1\otimes U_2,\alpha)\to (V\otimes W,\alpha)$ is completely bounded and $\|r_1\otimes r_2\|_{cb}\le \|r_1\|_{cb}\cdot \|r_2\|_{cb}$.
\end{enumerate}
\end{definition}
We   denote by $V\overset\alpha\otimes W$ the completion of  $\left(V\otimes W, \alpha\right)$.

This notion  is, in principle, less restrictive than the one introduced in   \cite[Definition 5.9]{BlePau}, which the authors  called ``uniform operator space tensor norm''.
Whenever the linear isomorphism determined  by  (\ref{matrizota}) (the so called algebraic {\sl shuffle}  isomorphism)  $M_p(V)\otimes M_q(W)\rightarrow M_{pq}(V\otimes W)$ extends to a complete contraction $M_p(V)\otimes_{\alpha} M_q(W)\rightarrow M_{pq}(V\otimes_{\alpha}  W)$, both notions coincide \cite{Wi}. That is the case of the three tensor norms defined above (projective, injective and Haagerup).
The proof  that these main examples satisfy the definition, as well as the fact that the projective tensor norm $\|\cdot\|_\wedge$ is the largest operator space tensor norm,  can be found in \cite{ER-libro}.

  Every operator space tensor norm determines, through $\nu$ in (\ref{linear isomorphisms}),  an operator space  bilinear ideal according to the following identification:  Given  $V$, $W$, $X$ operator spaces, let

  $$
\mathfrak{A}_\alpha(V\times W, X)\cong \mathcal{CB}(V\overset\alpha\otimes W,X).
$$

\begin{proposition}\label{prop maximal ideal}
Let $\alpha$ be an operator space  tensor norm.
Then $\mathfrak{A}_\alpha$ is an operator space bilinear ideal.
\end{proposition}

\begin{proof}
From the relation $\mathcal{CB}(V\overset\alpha\otimes W,X) \subset \mathcal{CB}(V\widehat\otimes W,X)$, it follows that $\mathfrak{A}_\alpha(V\times W, X)$  is a subspace of $\mathcal{JCB}(V\times W, X)$. Also, it is clear that all finite type continuous bilinear mappings belong to $\mathfrak{A}_\alpha(V\times W, X)$.

(a) Let $\phi\in M_n(\mathfrak{A}_\alpha(V\times W, X))$ then its linear associated $\widetilde{\phi}$ belongs to $M_n\left(\mathcal{CB}(V\overset\alpha\otimes W,X)\right)\cong \mathcal{CB}(V\overset\alpha\otimes W,M_n(X))$. This says that
$
\|\phi\|_{\mathfrak{A}_\alpha}=\|\widetilde{\phi}\|_{cb}=\sup_m\|\widetilde{\phi}_m\|.
$

The mapping $\widetilde{\phi}_m:M_m(V\overset\alpha\otimes W)\to M_m(M_n(X))$ has norm
\begin{equation*}
\|\widetilde{\phi}_m\| =  \sup\left\{|\widetilde{\phi}_m(u)|:\ u\in M_m(V\otimes W), \alpha(u)\le 1\right\}.
\end{equation*}
On the other hand, $\phi$ also belongs to $M_n(\mathcal{JCB}(V\times W, X))$ and it has an associated  matrix of  linear mappings  $\overline{\phi}\in M_n\left(\mathcal{CB}(V\widehat\otimes W,X)\right)\cong \mathcal{CB}(V\widehat\otimes W,M_n(X))$. This implies that
$$
\|\phi\|_{jcb}=\|\overline{\phi}\|_{cb}=\sup_m\|\overline{\phi}_m\|,
$$
and the mapping $\overline{\phi}_m:M_m(V\widehat\otimes W)\to M_m(M_n(X))$ has norm
\begin{equation*}
\|\overline{\phi}_m\| = \sup\left\{|\overline{\phi}_m(u)|:\ u\in M_m(V\otimes W), \|u\|_\wedge\le 1\right\}.
\end{equation*}

For each $u\in M_m(V\otimes W)$, $\widetilde{\phi}_m(u)=\overline{\phi}_m(u)$ and  $\alpha(u)\le \|u\|_\wedge$. Then,  for every $m$, $\|\overline{\phi}_m\|\le \|\widetilde{\phi}_m\|$,
and thus $\|\phi\|_{jcb}\le \|\phi\|_{\mathfrak{A}_\alpha}$.

\medskip

(b) For $\phi\in M_n(\mathfrak{A}_\alpha(V\times W, X))$, let  $\widetilde{\phi}\in M_n\left(\mathcal{CB}(V\overset\alpha\otimes W,X)\right)$ be its associated matrix of linear mappings. For any  $r_1\in \mathcal{CB}(U_1,V)$, $r_2\in \mathcal{CB}(U_2,W)$ and $s\in \mathcal{CB}(X,Y)$,   the following equality holds.
 \begin{eqnarray*}
\|s_n\circ\phi\circ (r_1,r_2)\|_{\mathfrak{A}_\alpha} & = & \|s_n\circ\widetilde{\phi}\circ (r_1\otimes r_2)\|_{cb}
\end{eqnarray*}
A direct computation gives the required inequality.
\end{proof}

\begin{example}
 Since $\mathcal{MB}(V\times W,X)\cong \mathcal{CB}(V\overset h\otimes W,X)$, from Proposition  \ref{prop maximal ideal}   we obtain that $\mathcal{MB}$ is an operator space bilinear ideal.
\end{example}

With similar arguments to those used to prove Proposition \ref{prop maximal ideal}, we obtain:

\begin{proposition}\label{prop: definition of an ideal through a tensor product}
Let $\alpha$ be an operator space tensor norm and $\mathfrak{B}$ be an operator space ideal of linear mappings.  Given the operator spaces $V$, $W$ and $X$, let  $\mathfrak{A}^{\mathfrak{B}}_\alpha(V\times W,X)$ be the operator space determined by  the identification
\begin{eqnarray} \label{bilinear tensor}
\mathfrak{A}^{\mathfrak{B}}_\alpha(V\times W,X)\cong \mathfrak{B}(V\overset\alpha\otimes W,X).
\end{eqnarray}
Then, $\mathfrak{A}^{\mathfrak{B}}_\alpha$  is an operator space bilinear ideal.
\end{proposition}

\section{Completely nuclear and completely integral bilinear mappings}\label{sect: compl nuclear compl integral}

In \cite[Sections 12.2 and 12.3]{ER-libro} the definitions of completely nuclear and completely integral linear mappings are presented.  We now introduce and study the analogous bilinear concepts. We will see that they define operator space  bilinear ideals. Theorem \ref{thm: main for integrals} provides a concrete identification of the  integral bilinear ideal as in (\ref{bilinear tensor}). On the contrary,  from Proposition \ref{prop: nucleares not dual de tensor}, it will  follow that the nuclear bilinear ideal can not be described in such a way.

   In order to properly define the  notion of nuclearity   in the context of bilinear mappings on operator spaces, we need to state first some natural mappings. Let
  $$
\Theta: (V^*\overset\vee\otimes W^*)\overset\vee\otimes X\hookrightarrow \mathcal{JCB}(V\times W,X)
$$ be   the  natural complete isometry obtained as a composition of the   natural complete isometries  $
V^*\overset\vee\otimes W^*\hookrightarrow (V\widehat\otimes W)^*$,  $\hspace{.5cm}(V\widehat\otimes W)^*\overset\vee\otimes X\hookrightarrow \mathcal{CB}(V\widehat\otimes W,X)\cong\mathcal{JCB}(V\times W,X)$\\ and
  $
(V^*\overset\vee\otimes W^*)\overset\vee\otimes X\hookrightarrow (V\widehat\otimes W)^*\overset\vee\otimes X
$ (see \cite[Proposition 8.1.2 and Proposition 8.1.5]{ER-libro}). Let
 $$\Phi: (V^*\widehat\otimes W^*)\widehat\otimes X\to (V^*\overset\vee\otimes W^*)\overset\vee\otimes X$$
 be the canonical  complete contraction and let
$$
\Psi=\Theta\circ\Phi:(V^*\widehat\otimes W^*)\widehat\otimes X\to\mathcal{JCB}(V\times W,X).
$$
With such a  $ \Psi$:
\begin{definition}
A bilinear mapping $\phi\in\mathcal{JCB}(V\times W,X)$ is \textbf{completely nuclear} if it belongs to the image of $\Psi$.  The operator space structure in the set  of completely nuclear bilinear mappings $\mathcal{N}(V\times W, X)$, is   given by the identification of the image of $\Psi$ with the quotient of its  domain by  its kernel. That is,
$$
\mathcal{N}(V\times W, X)\cong (V^*\widehat\otimes W^*)\widehat\otimes X/\ker\Psi.
$$
\end{definition}

\begin{proposition}
$\mathcal{N}$ is an operator space bilinear ideal.
\end{proposition}

\begin{proof}
By definition $\mathcal{N}(V\times W, X)$ is a linear subspace of $\mathcal{JCB}(V\times W,X)$ and the contention of finite type elements is plain.  The injective mapping
 $ \mathcal{N}(V\times W, X)\to\mathcal{JCB}(V\times W,X)$ induced on the quotient by the complete contraction $\Psi$, has norm less or equal than $\Psi$, and so, it is   again a complete contraction.  Hence, $\|\phi\|_{jcb}\le \|\phi\|_{\mathcal{N}}$ and  (a) is proved.

\medskip

(b) Let $\overline{\Psi}$ denote the quotient map induced by $\Psi$.  Given  $\phi\in M_n\left(\mathcal{N}(V\times W, X)\right)$, $r_1\in \mathcal{CB}(U_1,V)$, $r_2\in \mathcal{CB}(U_2,W)$ and $s\in \mathcal{CB}(X,Y)$, consider the following diagram:
\begin{equation*}
\xymatrix{M_n\left((V^*\widehat\otimes W^*)\widehat\otimes X\right)\ar[rr]^{\overline{\Psi}_n} \ar[d]^{\left((r_1^*\otimes r_2^*)\otimes s\right)_n} & & {M_n\left(\mathcal{N}(V\times W, X)\right)} \ar[d]^{} \\
 {M_n\left((U_1^*\widehat\otimes U_2^*)\widehat\otimes Y\right)} \ar[rr]^{\overline{\Psi}_n} & & M_n\left(\mathcal{N}(U_1\times U_2,Y)\right) , }
\end{equation*}
where the right vertical arrow is the mapping  $\phi\mapsto s_n\circ\phi\circ (r_1, r_2)$. It is immediate  to check that the mappings are well defined and that the  diagram commutes.  In particular, $s_n\circ\phi\circ (r_1, r_2)$ belongs to $M_n\left(\mathcal{N}(U_1\times U_2,Y)\right)$.
If $u\in M_n\left((V^*\widehat\otimes W^*)\widehat\otimes X\right)$ is such that $\overline{\Psi}_n(u)=\phi$ it holds
$$
s_n\circ\phi\circ (r_1, r_2)= s_n\circ\overline{\Psi}_n(u)\circ (r_1, r_2)=\overline{\Psi}_n\left(\left((r_1^*\otimes r_2^*)\otimes s\right)_n(u)\right).
$$
The estimate we are looking for follows from  the fact that the  inequality
\begin{eqnarray*}
\|s_n\circ\phi\circ (r_1, r_2)\|_{\mathcal{N}}&\le & \|\left((r_1^*\otimes r_2^*)\otimes s\right)_n(u)\|_{M_n\left((U_1^*\widehat\otimes U_2^*)\widehat\otimes Y\right)}
\end{eqnarray*}
holds for every $u$ such that $\overline{\Psi}_n(u)=\phi$.
\end{proof}

\begin{definition}
We say that a bilinear mapping $\phi\in\mathcal{JCB}(V\times W,X)$ is \textbf{completely integral} if
$$
\|\phi\|_{\mathcal{I}}=\sup\left\{\|\phi|_{F_1\times F_2}\|_{\mathcal{N}}:\ F_1\subset V,\ F_2\subset W\textrm{ of finite dimension}\right\}<\infty.
$$
\end{definition}
Let $\mathcal{I}(V\times W,X)$ be the space of all completely integral bilinear mappings from $V\times W$ to $X$. We consider in $\mathcal{I}(V\times W,X)$ the matrix norm given by
$$
\|\phi\|_{\mathcal{I}}=\sup\left\{\|\phi|_{F_1\times F_2}\|_{\mathcal{N}}:\ F_1\subset V,\ F_2\subset W\textrm{ of finite dimension}\right\},
$$
for every $\phi\in M_n\left(\mathcal{I}(V\times W,X)\right)$. It is easy to see that this norm endowed $\mathcal{I}(V\times W,X)$ with the structure of an operator space.

\begin{proposition}\label{prop: jcb < int <nuc} Let  $V,W$, $X$ be operator spaces and let  $\phi\in M_n\left(\mathcal{N}(V\times W,X)\right)$.
Then $$
\|\phi\|_{jcb}\le \|\phi\|_{\mathcal{I}}\le\|\phi\|_{\mathcal{N}}.
$$
The first inequality also holds for  $\phi\in M_n\left(\mathcal{I}(V\times W,X)\right)$.
\end{proposition}
\begin{proof} For  $\phi\in M_n\left(\mathcal{I}(V\times W,X)\right)$, consider finite dimensional spaces $F_1\subset V$ and $F_2\subset W$. Since
$\|\phi|_{F_1\times F_2}\|_{jcb}\le \|\phi|_{F_1\times F_2}\|_{\mathcal{N}}$ and
$$
\|\phi\|_{jcb}=\sup\left\{\|\phi|_{F_1\times F_2}\|_{jcb}:\ F_1\subset V,\ F_2\subset W\textrm{ of finite dimension}\right\}
$$

we obtain that
$$
\|\phi\|_{jcb}\le \|\phi\|_{\mathcal{I}}.
$$

Now, if $\phi\in M_n\left(\mathcal{N}(V\times W,X)\right)$ and we denote by $j_1:F_1\hookrightarrow V$ and $j_2:F_2\hookrightarrow W$ the canonical (completely contractive) embeddings, it is clear that
$$
\|\phi|_{F_1\times F_2}\|_{\mathcal{N}}= \|\phi\circ (j_1,j_2)\|_{\mathcal{N}}\le \|\phi\|_{\mathcal{N}}\cdot \|j_1\|_{cb}\cdot \|j_2\|_{cb}= \|\phi\|_{\mathcal{N}}.
$$
\end{proof}
\begin{proposition}
$\mathcal{I}$ is an operator space bilinear ideal.
\end{proposition}

\begin{proof}
By definition $\mathcal{I}(V\times W, X)$ is a linear subspace of $\mathcal{JCB}(V\times W,X)$. Finite type continuous bilinear maps are obviously contained in $\mathcal{I}(V\times W, X)$.
Condition (a)  was already proved above.

\medskip

(b) Let $\phi\in M_n\left(\mathcal{I}(V\times W, X)\right)$, $r_1\in \mathcal{CB}(U_1,V)$, $r_2\in \mathcal{CB}(U_2,V)$ and $s\in \mathcal{CB}(X,Y)$. For finite dimensional spaces $F_1\subset U_1$ and $F_2\subset U_2$ let $j_1:F_1\hookrightarrow U_1$ and $j_2:F_2\hookrightarrow U_2$ be the canonical (completely contractive) embeddings. We have
\begin{eqnarray*}
\left\|s_n\circ\phi\circ (r_1, r_2)|_{F_1\times F_2}\right\|_{\mathcal{N}}  = \|s_n\circ\phi\circ (r_1 j_1, r_2 j_2)\|_{\mathcal{N}} \le  \|s\|_{cb}\cdot \|\phi\|_{\mathcal{I}}\cdot \|r_1\|_{cb}\cdot \|r_2\|_{cb}.
\end{eqnarray*}
\end{proof}

   A pointwise limit of completely nuclear bilinear contractions
     is  not necessarily  completely nuclear, but it is always integral. This result is in the following two lemmas and  will be used several times. The statements given here
   are simpler than their linear analogues given in \cite[Lemma 12.2.7 and Lemma 12.3.1]{ER-libro}.

\begin{lemma}
Let $(\phi_\lambda)$ and $\phi$ in $M_n\left(\mathcal{N}(F_1\times F_2, M_m)\right)$, where $F_1$ and $F_2$ are finite dimensional operator spaces. Suppose that there exists a constant $C$ such that $\|\phi_\lambda\|_{M_n\left(\mathcal{N}(F_1\times F_2, M_m)\right)}\le C$ for all $\lambda$ and that $\phi_\lambda(x,y)\to \phi(x,y)$ for every $(x,y)\in F_1\times F_2$. Then, $\|\phi\|_{M_n\left(\mathcal{N}(F_1\times F_2, M_m)\right)}\le C$.
\end{lemma}

\begin{proof}
Take $\{x_1,\dots,x_k\}$ and $\{y_1, \dots, y_l\}$ vector bases of $F_1$ and $F_2$, respectively, and denote by $\{x_1^*,\dots,x_k^*\}$ and $\{y_1^*, \dots, y_l^*\}$ the corresponding dual bases. Since
$$
\phi_\lambda=\sum_{i,j} \phi_\lambda(x_i,y_j)\, x_i^*\otimes y_j^*\qquad\textrm{ and }\qquad \phi=\sum_{i,j} \phi(x_i,y_j)\, x_i^*\otimes y_j^*
$$ we have
\begin{eqnarray*}
\|\phi_\lambda-\phi\|_{M_n\left(\mathcal{N}(F_1\times F_2, M_m)\right)}&\le &\sum_{i,j} \|\phi_\lambda(x_i,y_j)-\phi(x_i,y_j)\|_{M_{n\cdot m}} \cdot \|x_i^*\otimes y_j^*\|_{\mathcal{N}(F_1\times F_2)}\\
& \le & \sum_{i,j} \|\phi_\lambda(x_i,y_j)-\phi(x_i,y_j)\|_{M_{n\cdot m}} \cdot \|x_i^*\|\cdot \|y_j^*\|\to 0.
\end{eqnarray*}
Hence, the result follows.
\end{proof}

\begin{lemma}\label{convergencia nuclear-integral}
Suppose that  $\phi\in M_n\left(\mathcal{JCB}(V\times W, M_m)\right)$ and that there exists a net $(\phi_\lambda)\subset M_n\left(\mathcal{N}(V\times W, M_m)\right)$ with
$$
\|\phi_\lambda\|_{M_n\left(\mathcal{N}(V\times W, M_m)\right)}\le C, \textrm{ for all }\lambda\quad \textrm{ and }\quad \phi_\lambda(v,w)\to \phi(v,w),\textrm{ for all }v\in V, w\in W.
$$ Then, $\phi$ belongs to $M_n\left(\mathcal{I}(V\times W, M_m)\right)$ and $\|\phi\|_{M_n\left(\mathcal{I}(V\times W, M_m)\right)}\le C$.
\end{lemma}

\begin{proof}
 For a given pair of finite dimensional  subspaces $F_1\subset V$ and $F_2\subset W$, the net $(\phi_\lambda|_{F_1\times F_2})$ and the map $\phi|_{F_1\times F_2}$ satisfy the hypothesis of the previous lemma. Thus, $\left\|\phi|_{F_1\times F_2}\right\|_{M_n\left(\mathcal{N}(V\times W, M_m)\right)}\le C$. This implies that $\phi$ is completely integral and $\|\phi\|_{M_n\left(\mathcal{I}(V\times W, M_m)\right)}\le C$.
\end{proof}

\

For the classes of completely nuclear and completely integral  mappings, it is necessary  to recall the linear definitions in order to   make precise  the relationship between bilinear mappings on operator spaces and linear mappings on operator space tensor products. A complete exposition of this topic is provided in \cite[Chapter 12]{ER-libro}. A linear mapping $\varphi:V\to W$ is said to be \textbf{completely nuclear}, $\varphi\in \mathcal{L_{\mathcal{N}}}(V,W)$, if it belongs to the image of the canonical completely contractive mapping
$$L_\Psi:V^*\widehat\otimes W\to V^*\overset\vee\otimes W \hookrightarrow \mathcal{CB}(V,W).
$$
The operator space structure of $\mathcal{L_{\mathcal{N}}}(V,W)$ is given by the identification
$$
\mathcal{L_{\mathcal{N}}}(V,W)\cong V^*\widehat\otimes W/\ker L_\Psi.
$$

 A linear mapping $\varphi:V\to W$ is said to be \textbf{completely integral}, $\varphi\in\mathcal{L_{\mathcal{I}}}(V,W)$, if the completely nuclear norms of all its restrictions to finite dimensional subspaces of $V$ are bounded. The operator space matrix norm on  $\mathcal{L_{\mathcal{I}}}(V,W)$ is given by
$$
\|\varphi\|_{M_n(\mathcal{L_{\mathcal{I}}}(V,W))}=\sup\left\{\|\varphi|_F\|_{\mathcal{L_{\mathcal{N}}}}:\ F\subset V \textrm{ of finite dimension}\right\},
$$ for each $\varphi\in M_n\left(\mathcal{L_{\mathcal{I}}}(V,W)\right)$.

  So, the relation we were seeking states the following:

\begin{theorem}\label{thm: main for integrals} For every three  operator spaces $V,W$ and $X$,
there is a complete isometry
$$
\mathcal{I}(V\times W, X)\cong \mathcal{L_{\mathcal{I}}}(V\overset\vee\otimes W,X).
$$
\end{theorem}

An analogous relation in the Banach space setting holds, and it  is crucial in the study of the bilinear integral mappings (see  \cite{Vi}).  The proof for  operator spaces is, however, quite more involved.

      We prove first  the  particular case of Theorem \ref{thm: main for integrals} when  $X$ is the finite dimensional operator space of  $n\times n $-matrices $M_n$.  The operator space dual/pre-dual of $M_n$ is the space   $T_n$  of $n\times n$-matrices where the norm is given by
$$
\|\alpha\|_{T_n}=\textrm{trace}(|\alpha|).
$$

\begin{remark}\label{Goldstine}
A version of ``Goldstine's theorem'' holds in operator spaces: If $u\in M_n(V^{**})$ with $\|u\|\le 1$, then there exists a net $(u_{\lambda})\in M_n(V)$ such that $\|u_\lambda\|\le 1$, for all $\lambda$ and $\varphi_n(u_\lambda)\to u(\varphi)$, for all $\varphi\in V^*$ (see \cite[Proposition 4.2.5]{ER-libro}).
\end{remark}

\begin{proposition}\label{prop: main for integrals into matrices}
There is a complete isometry
$ \,\,
\mathcal{I}(V\times W, M_n)\cong \mathcal{L_{\mathcal{I}}}(V\overset\vee\otimes W,M_n).
$
\end{proposition}

\begin{proof}
Since $M_n=T_n^*$ is a finite-dimensional operator space,  from   \cite[Corollary 12.3.4]{ER-libro} we get that there is a completely isometric identity
$$
\mathcal{L_{\mathcal{I}}}(V\overset\vee\otimes W,M_n)\cong \left((V\overset\vee\otimes W)\overset\vee\otimes T_n\right)^*.
$$
Thus, the result will be proved once we see that there is a complete isometry
$$
\mathcal{I}(V\times W, M_n)\cong \left((V\overset\vee\otimes W)\overset\vee\otimes T_n\right)^*.
$$

To that end, consider the following applications:
\begin{itemize}
\item $S:\mathcal{JCB}(V\times W, M_n)\to \left((V\widehat\otimes W)\widehat\otimes T_n\right)^*$, which is the canonical completely isometric isomorphism given by the identification
    $$
    \mathcal{JCB}(V\times W, M_n)\cong \mathcal{CB}(V\widehat\otimes W,M_n)\cong \left((V\widehat\otimes W)\widehat\otimes T_n\right)^*.
    $$
\item $\widehat{\Psi}:(V^*\widehat\otimes W^*)\widehat\otimes M_n \to \mathcal{N}(V\times W,M_n)$, the  quotient map.
\item $\Omega: (V^*\widehat\otimes W^*)\widehat\otimes M_n\to \left((V\overset\vee\otimes W)\overset\vee\otimes T_n\right)^*$, the linearization of the trilinear mapping
    \begin{eqnarray*}
    V^*\times W^*\times T_n^* & \to & (V\overset\vee\otimes W\overset\vee\otimes T_n)^*\\
    (v^*,w^*,\phi^*) &\mapsto & (v\otimes w\otimes \phi \mapsto v^*(v)w^*(w)\phi^*(\phi)),
    \end{eqnarray*} which is completely contractive.
\item $\Phi^*:\left((V\overset\vee\otimes W)\overset\vee\otimes T_n\right)^*\hookrightarrow \left((V\widehat\otimes W)\widehat\otimes T_n\right)^*$, which is the transpose mapping of $\Phi: (V\widehat\otimes W)\widehat\otimes T_n\to (V\overset\vee\otimes W)\overset\vee\otimes T_n$. Since $\Phi$ is a complete contraction and it has dense range, $\Phi^*$ results an injective complete contraction.
\end{itemize}
Replicating the argument of the linear case we use the previous mappings to construct a commutative diagram:
$$
\begin{array}{ccccc}
  \mathcal{N}(V\times W,M_n) & \subseteq & \mathcal{I}(V\times W,M_n) & \subseteq & \mathcal{JCB}(V\times W,M_n) \\
 \widehat{\Psi} \Big\uparrow &    & & & \Big\downarrow S\\
  (V^*\widehat\otimes W^*)\widehat\otimes M_n & \overset{\Omega}\longrightarrow & \left((V\overset\vee\otimes W)\overset\vee\otimes T_n\right)^* & \overset{\Phi^*}\longrightarrow &
  \left((V\widehat\otimes W)\widehat\otimes T_n\right)^*
\end{array}
$$

The injectivity of both $S_{|_\mathcal{N}}$ and $\Phi^*$  yields that $ \ker (\Omega)=\ker (\widehat{\Psi})$. This
allows us to define:
$$
S_{nuc}: \mathcal{N}(V\times W,M_n) \to \left((V\overset\vee\otimes W)\overset\vee\otimes T_n\right)^*
$$
in such a way  that $S_{nuc}\circ \widehat{\Psi}=\Omega$ and $ \Phi^*\circ S_{nuc}=S_{|_\mathcal{N}}$. The mapping  $S_{nuc}$ is a complete contraction.

Let us suppose now that $\phi\in \mathcal{I}(V\times W,M_n)$ with $\|\phi\|_{\mathcal{I}(V\times W,M_n)}\le 1$. We want to see that $S(\phi)$ is continuous with respect to the injective tensor norm of $(V\otimes W)\otimes T_n$. Given $u\in (V\otimes W)\otimes T_n$ with $\|u\|_\vee\le 1$, there exist finite-dimensional spaces $V_u\subset V$ and $W_u\subset W$ such that $u\in (V_u\otimes W_u)\otimes T_n$. Let us call $j_{V_u}:V_u\hookrightarrow V$ and $j_{W_u}:W_u\hookrightarrow W$ the canonical inclusions, then
$$
\langle S(\phi), u\rangle = \langle S_{nuc}(\phi\circ (j_{V_u},j_{W_u})), u\rangle.
$$
Therefore,
\begin{eqnarray*}
|\langle S(\phi), u\rangle | &\le & \| S_{nuc}(\phi\circ (j_{V_u},j_{W_u}))\|_{\left((V_u\overset\vee\otimes W_u)\overset\vee\otimes T_n\right)^*}\cdot \|u\|_{(V_u\overset\vee\otimes W_u)\overset\vee\otimes T_n}\\
&\le & \|\phi\circ (j_{V_u},j_{W_u})\|_{\mathcal{N}(V_u\times W_u, M_n)}\cdot \|u\|_\vee\\
&\le & \|\phi\|_{\mathcal{I}(V\times W,M_n)} \le 1.
\end{eqnarray*}

Thus,  $S$  determines a contractive mapping
$$
S_{int}:\mathcal{I}(V\times W,M_n)\to \left((V\overset\vee\otimes W)\overset\vee\otimes T_n\right)^*.
$$
Through a similar argument it can be seen that $S_{int}$ is also a complete contraction.

Let us show now that $S_{int}$ is a complete isometry. For that, get $\phi\in M_m\left(\mathcal{I}(V\times W,M_n)\right)$ such that $\|(S_{int})_m(\phi)\|_{M_m\left(\left((V\overset\vee\otimes W)\overset\vee\otimes T_n\right)^*\right)}\le 1$. We have to prove that $\|\phi\|_{M_m\left(\mathcal{I}(V\times W,M_n)\right)}\le 1$.

Since $(S_{int})_m(\phi)\in M_m\left(\left((V\overset\vee\otimes W)\overset\vee\otimes T_n\right)^*\right)\cong \mathcal{CB}((V\overset\vee\otimes W)\overset\vee\otimes T_n, M_m)$ and $(V\overset\vee\otimes W)\overset\vee\otimes T_n\hookrightarrow \mathcal{JCB}(V^*\times W^*,T_n)$ is a complete isometry, by Remark \ref{extensiones},  $(S_{int})_m(\phi)$ extends to $\widetilde{(S_{int})_m(\phi)}\in \mathcal{CB}(\mathcal{JCB}(V^*\times W^*,T_n), M_m)$ preserving the norm. Now, we have completely isometric identifications
\begin{eqnarray*}
\mathcal{CB}(\mathcal{JCB}(V^*\times W^*,T_n), M_m)&\cong & \mathcal{CB}(\mathcal{CB}(V^*\widehat\otimes W^*,T_n), M_m)\cong \mathcal{CB}\left(\left((V^*\widehat\otimes W^*)\widehat\otimes M_n\right)^*,M_m\right)\\
&\cong & M_m\left(\left((V^*\widehat\otimes W^*)\widehat\otimes M_n\right)^{**}\right),
\end{eqnarray*}
and we thus know that $\|\widetilde{(S_{int})_m(\phi)}\|_{M_m\left(\left((V^*\widehat\otimes W^*)\widehat\otimes M_n\right)^{**}\right)}\le 1$. Hence, by Remark \ref{Goldstine}, there exists a net $(u_\lambda)$ in $M_m\left((V^*\widehat\otimes W^*)\widehat\otimes M_n\right)$ with $\|u_\lambda\|\le 1$ such that, for all $\varphi\in \left((V^*\widehat\otimes W^*)\widehat\otimes M_n\right)^*$,
$$
\varphi_m(u_\lambda)\to \widetilde{(S_{int})_m(\phi)}(\varphi).
$$

In particular, for any $v\in V$, $w\in W$ and $\alpha\in T_n$,
$$
((v\otimes w)\otimes\alpha)_m(u_\lambda)\to \widetilde{(S_{int})_m(\phi)}((v\otimes w)\otimes\alpha)= (S_{int})_m(\phi)((v\otimes w)\otimes\alpha).
$$

Looking into the coordinates of this matrix limit, with the notation $u_\lambda=(u_\lambda^{k,l})_{k,l}$ and $\phi=(\phi^{k,l})_{k,l}$, we obtain
$$
\langle \widehat{\Psi}(u_\lambda^{k,l})(v,w),\alpha\rangle=((v\otimes w)\otimes\alpha)(u_\lambda^{k,l})\to S_{int}(\phi^{k,l})((v\otimes w)\otimes \alpha)=\langle \phi^{k,l}(v,w),\alpha\rangle,
$$ for every $(v,w)\in V\times W$, $\alpha\in T_n$ and $k,l\in\{1,\dots,m\}$.
Thus, for each pair $(v,w)\in V\times W$, the net $\left(\widehat{\Psi}(u_\lambda^{k,l})(v,w)\right)_{k,l}$ converges weakly  to $\phi^{k,l}(v,w)$. Being $M_n$ a finite dimensional space, this convergence turns out to be strong and now we can also forget the coordinates and look at the whole picture again. So we have $\widehat{\Psi}_m(u_\lambda)(v, w) \to \phi(v,w)$, for all $(v,w)\in V\times W$.

Since $\widehat{\Psi}$ is a complete contraction, we know $\|\widehat{\Psi}_m(u_\lambda)\|_{M_m\left(\mathcal{N}(V\times W,M_n)\right)}\le 1$ and with an appealing to
Lemma \ref{convergencia nuclear-integral} we derive that $\|\phi\|_{M_m\left(\mathcal{I}(V\times W,M_n)\right)}\le 1$.

It only remains to prove that $S_{int}$ is surjective. Let $f\in \left((V\overset\vee\otimes W)\overset\vee\otimes T_n\right)^*$. The surjectivity of $S$ tells us that there exists $\phi\in \mathcal{JCB}(V\times W,M_n)$ such that $\Phi^*(f)=S(\phi)$. Moreover,  for finite dimensional spaces $F_1\in V$ and $F_2\in W$ with canonical inclusions $j_1:F_1\hookrightarrow V$ and $j_2:F_2\hookrightarrow W$ it holds
$$
\Phi^*(f\circ (j_1,j_2))=S(\phi\circ (j_1,j_2)).
$$
Since $\phi\circ (j_1,j_2)$ belongs to $\mathcal{N}(F_1\times F_2, M_n)\cong \mathcal{I}(F_1\times F_2, M_n)$ it is clear that $S_{int}(\phi\circ (j_1,j_2))= f\circ (j_1,j_2)$.

Hence,
\begin{eqnarray*}
\|\phi\circ (j_1,j_2)\|_{\mathcal{N}(F_1\times F_2, M_n)} & = & \|\phi\circ (j_1,j_2)\|_{\mathcal{I}(F_1\times F_2, M_n)} = \|S_{int}(\phi\circ (j_1,j_2))\|\\
& =  & \|f\circ (j_1,j_2)\| \le \|f\|.
\end{eqnarray*}

Thus, $\phi\in \mathcal{I}(V\times W, M_n)$ with $\|\phi\|_{\mathcal{I}(V\times W, M_n)}\le \|f\|$.
\end{proof}

Now we can prove  the general result $
\mathcal{I}(V\times W, X)\cong \mathcal{L_{\mathcal{I}}}(V\overset\vee\otimes W,X):
$

\begin{proof}[Proof of Theorem \ref{thm: main for integrals}]
Let $\phi\in \mathcal{I}(V\times W, X)$ and consider the associated linear application
$$
L_\phi:V\otimes W\to X.
$$ We begin by proving that $L_\phi$ is completely bounded from $(V\otimes W, \vee)$ to $X$. This will allows us to extend $L_\phi$ to $ V\overset\vee\otimes W$. For that, we need to find a common bound for the norms of the mappings
$$
(L_\phi)_n: M_n(V\otimes W,\vee)\to M_n(X).
$$
Let $u\in M_n(V\otimes W)$. By \cite[Lemma 2.3.4]{ER-libro}, there exists $\xi\in \mathcal{CB}(X,M_n)$ with $\|\xi\|_{cb}\le 1$ satisfying
$$
\|(L_\phi)_n(u)\|_{M_n(X)} = \|\xi_n\left((L_\phi)_n(u)\right)\|_{M_n(M_n)}=\|(\xi\circ L_\phi)_n(u)\|_{M_n(M_n)}= \|(L_{\xi\circ\phi})_n(u)\|_{M_n(M_n)}.
$$
Since $\xi\circ\phi:V\times W\to M_n$ is completely integral, we know from  Proposition \ref{prop: main for integrals into matrices} that $L_{\xi\circ\phi}$ belongs to $\mathcal{L_{\mathcal{I}}}(V\overset\vee\otimes W,M_n)$. Thus, $L_{\xi\circ\phi}\in\mathcal{CB}(V\overset\vee\otimes W,M_n)$ and therefore,
\begin{equation*}
\|(L_{\xi\circ\phi})_n(u)\|_{M_n(M_n)} \le \|L_{\xi\circ\phi}\|_{cb}\cdot \|u\|_{M_n(V\overset\vee\otimes W)}\le \|\xi\|_{cb}\cdot  \|\phi\|_{\mathcal{I}(V\times W,X)}\cdot \|u\|_{M_n(V\overset\vee\otimes W)}.
\end{equation*}

This yields that $L_\phi\in \mathcal{CB}(V\overset\vee\otimes W,X)$. Let us prove now that, indeed, given
 $\phi\in M_n\left(\mathcal{I}(V\times W, X)\right)$,  $L_\phi$ belongs to  $M_n\left(\mathcal{L}_{\mathcal{I}}(V\overset\vee\otimes W,X)\right)$. To that end we need to compute the nuclear norms of its restrictions to finite dimensional spaces. Let $F\subset V\overset\vee\otimes W$ be a finite dimensional subspace. There exist finite dimensional subspaces $F_1\in V$ and $F_2\in W$ such that $F\subset F_1\overset\vee\otimes F_2$. The complete isometry $(F_1\overset\vee\otimes F_2)^*\cong F_1^*\widehat\otimes F_2^*$ (see, for instance, \cite[(15.4.1)]{ER-libro}) yields that $\mathcal{N}(F_1\times F_2, X)\cong \mathcal{L_\mathcal{N}}(F_1\overset\vee\otimes F_2,X)$.  Thus,
$$
\|L_\phi|_F\|_{M_n\left(\mathcal{L}_{\mathcal{N}}(F,X)\right)}\le \left\|L_\phi|_{F_1\overset\vee\otimes F_2}\right\|_{M_n\left(\mathcal{L}_{\mathcal{N}}(F_1\overset\vee\otimes F_2,X)\right)}= \|\phi|_{F_1\times F_2}\|_{M_n\left(\mathcal{N}(F_1\times F_2,X)\right)}\le \|\phi\|_{M_n\left(\mathcal{I}(V\times W, X)\right)}
$$
Hence, it follows that $L_\phi\in M_n\left(\mathcal{L}_{\mathcal{I}}(V\overset\vee\otimes W,X)\right)$ and $\|L_\phi\|_{M_n\left(\mathcal{L}_{\mathcal{I}}(V\overset\vee\otimes W,X)\right)}\le \|\phi\|_{M_n\left({\mathcal{I}}(V\times W, X)\right)}$.

To prove the opposite contention,  consider   $L\in M_n\left(\mathcal{L}_{\mathcal{I}}(V\overset\vee\otimes W,X)\right)$.  It is plain to see that $L$ is $L_\phi$, for some $\phi\in M_n\left(\mathcal{JCB}(V\times W, X)\right). $ The same argument as above shows that for any finite dimensional subspaces $F_1\in V$ and $F_2\in W$,
$$
\|\phi|_{F_1\times F_2}\|_{M_n\left(\mathcal{N}(F_1\times F_2,X)\right)}= \left\|L_\phi|_{F_1\overset\vee\otimes F_2}\right\|_{M_n\left(\mathcal{L}_{\mathcal{N}}(F_1\overset\vee\otimes F_2,X)\right)}\le \|L_\phi\|_{M_n\left(\mathcal{L}_{\mathcal{I}}(V\overset\vee\otimes W,X)\right)}.
$$
Consequently, $\phi\in M_n\left(\mathcal{I}(V\times W, X)\right)$ and $ \|\phi\|_{M_n\left(\mathcal{I}(V\times W, X)\right)}\le \|L_\phi\|_{M_n\left(\mathcal{L}_{\mathcal{I}}(V\overset\vee\otimes W,X)\right)}$.
\end{proof}

 \subsection*{The scalar valued case}
 Let $V$ and $W$ be operator spaces and let $\nu$ be the linear isomorphism in (\ref{linear isomorphisms}).  As a corollary of   Theorem \ref{thm: main for integrals} we have that  $\nu$ induces the following  complete isometry:

\begin{proposition}\label{prop: integrales=dual del inyectivo}
 $
\mathcal{I}(V\times W)\cong (V\overset\vee\otimes W)^*.
$
\end{proposition}

In contrast,  in the case of the nuclear  bilinear ideal we have:
\begin{proposition}\label{prop: nucleares not dual de tensor}
The following are equivalent:
\begin{enumerate}
\item[(i)] There exists an operator space tensor norm $\alpha$ such that $\mathcal{N}(V\times W)\cong (V\overset\alpha\otimes W)^*$.
\item[(ii)] $\mathcal{N}(V\times W)=\mathcal{I}(V\times W)$.
\end{enumerate}
In this case, $\alpha$ coincides with  the injective operator space tensor norm.
\end{proposition}
\begin{proof}   (i) follows from (ii) by  Proposition \ref{prop: integrales=dual del inyectivo}. To prove the other implication,
recall that $\|\cdot\|_{\vee} \leq \|\cdot\|_{\alpha}$ for any operator space tensor norm $\alpha$.  Thus, if (i) holds for some $\alpha$, then $\mathcal{I}(V\times W)\cong (V\overset\vee\otimes W)^*\subset (V\overset\alpha\otimes W)^*\cong \mathcal{N}(V\times W)$.
 \end{proof}

 It is worth noticing that there are examples of  completely  integral scalar valued bilinear mappings which are not completely nuclear (see Example \ref{A completely integral bilinear form which is not completely nuclear}). Thus, the completely nuclear bilinear ideal is not of the type described in Proposition \ref{prop: definition of an ideal through a tensor product}.

Something more can be said about a tensorial representation of $\mathcal N(V\times W)$. First, recall the following definition

\begin{definition}
An operator space $V$ is said to have the \textbf{operator space approximation property (OAP)} if for every $u\in \mathcal K(H)\overset\vee\otimes V$ and for every $\varepsilon >0$ there exists a finite rank mapping $T$ on $V$ such that $\|u-(I\otimes T)(u)\|< \varepsilon$.
\end{definition}

By \cite[Theorem 11.2.5]{ER-libro}, $V$ has OAP if and only if the canonical inclusion $V\widehat\otimes W\hookrightarrow V \overset\vee\otimes W$ is one-to-one, for every operator space $W$ (or just for $V^*$). Recall that the standard translation of this result to the Banach space setting was also valid. As a direct consequence we can state the following:

\begin{proposition} \label{OAP}
If $V^*$ or $W^*$ has OAP then there is a complete isometry:
$$
\mathcal N(V\times W)\cong V^*\widehat\otimes W^*.
$$
\end{proposition}

As an example we can consider a reflexive operator space $V$ such that its dual $V^*$, looked as a Banach space  has the (Banach) approximation property but as an operator space $V^*$ has not OAP (see \cite{Ari,OikRic} for examples of such spaces). In this case, the space of (Banach) nuclear bilinear forms on $V\times V^*$ has a canonical representation as a projective tensor product while the space of completely nuclear bilinear forms has not:
$$
\mathcal N^{B}(V\times V^*)\cong V^*\otimes_{\pi} V^{**}\qquad\textrm{ and }\qquad \mathcal N(V\times V^*)\not\cong V^*\widehat\otimes V^{**}.
$$

\begin{remark}
The argument in Proposition \ref{OAP} can be easily extended to the vector valued case. Hence, we have
$$
\mathcal N(V\times W,X)\cong (V^*\widehat\otimes W^*)\widehat\otimes X,
$$ whether two of the three spaces $V^*$, $W^*$ and $X$ have OAP.
\end{remark}

\bigskip

Looking at the equivalence $\mathcal{JCB}(V\times W)\simeq \mathcal{CB}(V,W^*) $ and taking into account the situation in the Banach space setting,
 we question about the existence of an operator space identification for completely nuclear bilinear/linear mappings  and for completely integral bilinear/linear mappings.

For the nuclear case, a careful look to the definitions of the spaces of completely nuclear bilinear and linear mappings, easily gives the following.

  \begin{proposition}\label{prop: nucleares bilineales=nucleares lineales into dual}
  $
\mathcal{N}(V\times W)\cong \mathcal{L_{\mathcal{N}}}(V,W^*).
$
  \end{proposition}

The situation for completely integral mappings is quite different: since $\mathcal{L_{\mathcal{I}}}(V, W^*)$ is not always completely isometric to $(V\overset\vee\otimes W)^*$ \cite[Section 12.3]{ER-libro} then neither the spaces $\mathcal{I}(V\times W)$ and $\mathcal{L_{\mathcal{I}}}(V, W^*)$ are  always completely isometric. In the Banach space setting, the space of integral bilinear forms from two Banach spaces is isometrically isomorphic to the space of integral linear mappings from one of the spaces to the  dual of the other (see, for instance, \cite[Proposition 3.22]{Ryan-libro}). The hidden reason behind this different behavior is the  Principle of Local Reflexivity, which is valid for every Banach space  while its operator space version does not always hold (see \cite[Section 14.3]{ER-libro} or \cite[Definition 18.1]{Pisier-libro} for a precise definition). Indeed,  \cite[Theorem 14.3.1]{ER-libro} along with Proposition \ref{prop: integrales=dual del inyectivo} give us the statement below.
\begin{proposition}\label{prop: integrales bilineales neq integrales lineales into dual}
Let $W$ be an operator space.Then the following are equivalent:
\begin{enumerate}
\item[(i)] $W$ is locally reflexive.
\item[(ii)] For every operator space $V$, there is a complete isometry $\mathcal{I}(V\times W)\cong\mathcal{L_{\mathcal{I}}}(V, W^*)$.
\end{enumerate}
\end{proposition}

\section{Completely extendible bilinear mappings}\label{sect: compl extendibles}

Within the scope of Banach spaces, the non-validity of a Hahn-Banach theorem for multilinear mappings and homogeneous polynomials motivates the study of the `extendible' elements (those that can be extended to any superspace). We propose and study here a version of this concept for bilinear mappings between operator spaces. Our approach  was strongly inspired by the results and arguments of \cite{Car} (see also \cite{KR}).

\begin{definition}
A mapping $\phi\in \mathcal{JCB}(V\times W,Z)$ is \textbf{completely extendible} if for any operator spaces $X$ and $Y$ such that $V\subset X$, $W\subset Y$ there exists a jointly completely bounded extension $\overline{\phi}:X\times Y\to Z$ of $\phi$.
\end{definition}

By the Representation Theorem for operator spaces (see, for instance \cite[Theorem 2.3.5]{ER-libro}), any operator space can be seen, through a complete isometry, as a subspace of certain $\mathcal{L}(H)$.

Given $V$ and $W$, let us denote the complete isometries that realize these spaces by
$$
\Omega_V:V\to\mathcal L(H_V)\qquad\textrm{ and }\qquad \Omega_W:W\to\mathcal L(H_W).
$$

Following  the idea of \cite[Theorem 3.2]{Car}, we obtain:

\begin{proposition} \label{prop:extendible a L(H)}
A jointly completely bounded mapping $\phi:V\times W\to Z$ is extendible if and only if it can be extended to $\mathcal L(H_V)\times \mathcal L(H_W)$. In this case, if $\phi_0$ is such an extension, then for every $X\supset V$ and $Y\supset W$ there exists an extension  $\overline{\phi}:X\times Y\to Z$ with $\|\overline{\phi}\|_{jcb}\le \|\phi_0\|_{jcb}$.
\end{proposition}

\begin{proof} Let $\phi_0: \mathcal L(H_V)\times \mathcal L(H_W)\to Z$ be an extension of $\phi$.
By Remark \ref{extensiones}, $\Omega_V$ and $\Omega_W$  have complete contractive extensions $\overline{\Omega}_V:X\to \mathcal L(H_V)$ and $\overline{\Omega}_W:Y\to \mathcal L(H_W)$. Then, $\overline{\phi}:X\times Y\to Z$ given by
$$
\overline{\phi}(x,y)=\phi_0(\overline{\Omega}_V(x), \overline{\Omega}_W(y)),\quad\textrm{for all }x\in X,\ y\in Y,
$$ extends $\phi$ and
$$
\|\overline{\phi}\|_{jcb}\le \|\phi_0\|_{jcb}\cdot \|\overline{\Omega}_V\|_{cb}\cdot \|\overline{\Omega}_W\|_{cb} = \|\phi_0\|_{jcb}.
$$
\end{proof}

Let
$$
\mathcal{E}(V\times W, Z)=\left\{\phi\in \mathcal{JCB}(V\times W, Z):\ \phi \textrm{ is extendible}\right\}.
$$

It is clear that $\mathcal{E}(V\times W, Z)$ is a subspace of $\mathcal{JCB}(V\times W, Z)$. Moreover, it is an operator space if we consider the following norm: for each $\phi\in M_n\left(\mathcal{E}(V\times W, Z)\right)$, let $\|\phi\|_{\mathcal{E}}$  be the infimum of the numbers $C>0$ such that for all $X\supset V$ and $Y\supset W$ there exists $\overline{\phi}\in M_n\left(\mathcal{JCB}(X\times Y, Z)\right)$ which extends $\phi$,  $\|\overline\phi\|_{jcb}\le C$.  The  previous proposition tells us that we can define equivalently
$$
\|\phi\|_{\mathcal{E}} =\inf\left\{ \|\phi_0\|_{jcb}:\ \phi_0\textrm{ extension of }\phi\textrm{ to }M_n\left(\mathcal{JCB}(\mathcal L(H_V)\times\mathcal L(H_W),Z)\right)\right\}.
$$

\begin{proposition} \label{prop:extendible es ideal}
$\mathcal E$ is an operator space bilinear ideal.
\end{proposition}
\begin{proof}  Since continuous functionals are completely extendible, it is clear that all finite type continuous bilinear mappings belong to this subspace.

(a) For any $\phi\in M_n\left(\mathcal E(V\times W,Z)\right)$ we know that $\|\phi\|_{jcb}\le \|\phi_0\|_{jcb}$ for every extension $\phi_0\in M_n\left(\mathcal{JCB}(\mathcal L(H_V)\times\mathcal L(H_W),Z)\right)$. Thus,   $\|\phi\|_{jcb}\le \|\phi\|_{\mathcal{E}}$.

(b) Consider  $\phi\in M_n\left(\mathcal E(V\times W,Z)\right)$, $r_1\in\mathcal{CB}(U_1,V)$,  $r_2\in\mathcal{CB}(U_2,W)$ and $s\in\mathcal{CB}(Z,Y)$.
  Since $\phi$ is a matrix of completely extendible maps, given $\varepsilon >0$, there exists  an extension $\phi_0\in M_n\left(\mathcal{JCB}(\mathcal{L}(H_{V})\times\mathcal{L}(H_{W}), Z)\right)$ such that $\|\phi_0\|_{jcb}\le \|\phi\|_{\mathcal{E}}+\varepsilon$.

 According to Remark \ref{extensiones},  let  $R_1:\mathcal{L}(H_{U_1}) \rightarrow  \mathcal{L}(H_{V})$ and $ R_2:\mathcal{L}(H_{U_2}) \rightarrow  \mathcal{L}(H_{W})$ be completely bounded extensions of
$r_1$ and $r_2$, respectively, with $\|r_1\|_{cb}=\|R_1\|_{cb}$ and $\|r_2\|_{cb}=\|R_2\|_{cb}$. Then,    $s_n\circ\phi_0\circ(R_1, R_2)$ is an extension of $s_n\circ\phi\circ (r_1, r_2)$ to $\mathcal{L}(H_{U_1})\times\mathcal{L}(H_{U_2})$  and
$$
\|s_n\circ\phi_0\circ(R_1, R_2)\|_{jcb}\le \|s\|_{cb}\cdot\|\phi_0\|_{jcb}\cdot \|R_1\|_{cb}\cdot \|R_2\|_{cb}\le \|s\|_{cb}\cdot\left(\|\phi\|_{\mathcal{E}}+\varepsilon\right)\cdot \|r_1\|_{cb}\cdot \|r_2\|_{cb}.
$$
Therefore, $s_n\circ\phi\circ (r_1,r_2)\in M_n\left(\mathcal E(U_1\times U_2,Z)\right)$ and $\|s_n\circ\phi\circ (r_1, r_2)\|_{\mathcal E}\le \|s\|_{cb}\cdot\|\phi\|_{\mathcal E}\cdot\|r_1\|_{cb}\cdot \|r_2\|_{cb}$.
\end{proof}

Motivated by what is done in the Banach space setting (see \cite[Corollary 3.9]{Car} or \cite[Proposition 3]{KR}), we now  define an operator space   tensor norm  $\eta$  such that for any $V, W$,  the dual operator space $(V\overset\eta\otimes W)^*$ coincides with the  scalar-valued completely extendible bilinear mappings $\mathcal E(V\times W)$. To that end,
 consider the tensor product of the canonical operator space inclusions where the range is endowed with the  operator space projective tensor norm:
$$
\Omega_V\otimes\Omega_W:V\otimes W\to \mathcal L(H_V)\widehat\otimes\mathcal L(H_W).
$$
Let $\eta$   be the operator space tensor norm in $V\otimes W$ induced by this application. Thus, for any $u\in M_n\left(V\otimes W\right)$,
$$
\eta(u)=\|(\Omega_V\otimes\Omega_W)_n(u)\|_\wedge.
$$

It is plain to see that $\eta$ is an operator space matrix norm that does not depend on the representations of $\Omega_V$ and $\Omega_W$ but just on the operator space structure of $V$ and $W$. Also, since $\Omega_V$ and $\Omega_W$ are complete isometries it easily follows that $\eta$ is a cross matrix norm. Moreover, it can be proved evidently that $\eta$ is an operator space tensor norm according to Definition \ref{def: operator space tensor norm}.

Let  $V\overset\eta\otimes W$ denote the completion of $(V\otimes W, \eta)$.

\begin{proposition}\label{prop: extendible-tensor}
There is a complete isometry
$$
\left(V\overset\eta\otimes W\right)^*\cong \mathcal E(V\times W).
$$
\end{proposition}

\begin{proof}
Let $\varphi\in \left(V\overset\eta\otimes W\right)^*$ and denote by $\phi$ the associated bilinear form, $\phi:V\times W\to \mathbb{C}$. Since $V\overset\eta\otimes W\hookrightarrow \mathcal L(H_V)\widehat\otimes\mathcal L(H_W)$ is a complete isometry, we can see $V\overset\eta\otimes W$ as a subspace of $\mathcal L(H_V)\widehat\otimes\mathcal L(H_W)$. By Remark \ref{extensiones}, $\varphi$ can be extended to $\varphi_0:\mathcal L(H_V)\widehat\otimes\mathcal L(H_W)\to\mathbb{C}$ with $\|\varphi_0\|_{cb}=\|\varphi\|_{cb}$. It is easy to see that the bilinear map $\phi_0:\mathcal L(H_V)\times\mathcal L(H_W)\to\mathbb{C}$ associated to $\varphi_0$ is an extension of $\phi$. Also,
$$
\|\phi_0\|_{jcb}=\|\varphi_0\|_{cb}=\|\varphi\|_{cb}.
$$ Then, $\phi$ is completely extendible and $\|\phi\|_{\mathcal{E}}\le \|\varphi\|$.

Reciprocally, let $\phi\in\mathcal E(V\times W)$ and denote its linear associated by $\varphi:V\otimes W\to \mathbb{C}$. Let $\phi_0: \mathcal L(H_V)\times\mathcal L(H_W)\to\mathbb{C}$ be an extension of $\phi$ and consider its linear associated $\varphi_0\in \left(\mathcal L(H_V)\widehat\otimes\mathcal L(H_W)\right)^*$. Thus, for each $u\in V\otimes W$,
$$
|\varphi(u)|=|\varphi_0(\Omega_V\otimes\Omega_W)(u)|\le \|\varphi_0\|_{cb}\cdot \|(\Omega_V\otimes\Omega_W)(u)\|_\wedge=\|\varphi_0\|_{cb}\cdot\|u\|_\eta.
$$ This implies that $\varphi$ is $\eta$-continuous and so it can be extended continuously to $V\overset\eta\otimes W$. Hence, $\varphi\in \left(V\overset\eta\otimes W\right)^*$ with $\|\varphi\|\le \|\phi\|_{\mathcal{E}}$.

The isometry between $\left(V\overset\eta\otimes W\right)^*$ and $\mathcal E(V\times W)$ is now proved and a similar argument shows that the isometry is complete.
\end{proof}

\section{The symmetrized multiplicatively bounded bilinear ideal}\label{sect: SMB ideal}

   Given a bilinear mapping  $\phi :V\times W \rightarrow Z$, its transposed  $\phi^t  :W\times V\rightarrow Z$ is defined by the relation $\phi^t(w,v)=\phi(v,w)$.
   We will say that an operator space  bilinear ideal $\mathfrak{A}$ is \textbf{symmetric} when satisfies that if  $\phi\in \mathfrak{A}(V\times W,Z)$ then   $ \phi^t\in \mathfrak{A}(W\times V,Z)$  with $\|\phi\|_{\mathfrak A}= \|\phi^t\|_{\mathfrak A}$.

   The bilinear ideals  $\mathcal{JCB}, \,  \mathcal{N},  \, \mathcal{I} $ and $ \mathcal{E}$ are clearly symmetric, while $\mathcal{MB}$ is not (see Example \ref{mb no i}).

   \begin{definition}
A bounded bilinear mapping $\phi: V\times W \rightarrow Z$ is \textbf{symmetrized multiplicatively bounded}, $\phi \in \mathcal{SMB}(V\times W, Z)$
  if it can be decomposed as  $\phi =\phi_1 + \phi_2$ with $\phi_1\in \mathcal{MB}(V\times W,Z)$ and $ \phi_2^t\in \mathcal{MB}(W\times V,Z)$.
  \end{definition}

    The space $\mathcal{SMB}(V\times W, Z)$ is equiped with  an operator space structure through the identification with  the sum  $ \mathcal{MB}(V\times W,Z)+ \ ^t\mathcal{MB}(W\times V,Z)$  in the sense of operator spaces interpolation theory
   (see \cite[Chapter 2]{Pisier AMS}). In this way, the norm of a  matrix $\phi \in M_n(\mathcal{SMB}(V\times W; Z))$  is given by
  $$
\|\phi\|_{{smb}} =\inf \left\{ \|(\phi_1,\phi_2) \|_{M_n(\mathcal{MB}(V\times W,Z)\oplus_1  {{^t}\mathcal{MB}}(W\times V,Z))}: \phi = \phi_1 + \phi_2  \right\}.$$

\begin{proposition} \label{prop: SMB es ideal}
$\mathcal{SMB}$ is a symmetric  operator space bilinear ideal.
\end{proposition}
\begin{proof} By means of \cite[Proposition 2.1]{Pisier AMS} it is easy to see that whenever $\mathfrak A_1$ and $\mathfrak A_2$ are operator space bilinear ideals then the same holds for $\mathfrak A_1 + \mathfrak A_2$. Hence, this is valid for $\mathcal{SMB}=\mathcal{MB} + \ ^t\mathcal{MB}$.
\end{proof}

We denote by $(V\overset{h}\otimes W)\cap (W\overset{h}\otimes V)$ the set of elements $u$ in $V\overset{h}\otimes W$ such that $u^t$ belongs to $W\overset{h}\otimes V$. Appealing again to interpolation theory, we can see  $(V\overset{h}\otimes W)\cap (W\overset{h}\otimes V)$ as an operator space with the structure inherited by the canonical inclusion in $(V\overset{h}\otimes W)\oplus_\infty (W\overset{h}\otimes V)$.

The completely isometric identity $(X\cap Y)^*\cong X^*+Y^*$ \cite[page 23]{Pisier AMS} applied to our case says:
$$
\mathcal{SMB}(V\times W)\cong \Big((V\overset{h}\otimes W)\cap (W\overset{h}\otimes V)\Big)^*,
$$ completely isometrically. In the vector-valued case, there is also some interplay between the space of symmetrized multiplicatively bounded bilinear mappings and the intersection of both Haagerup tensor products:

\begin{proposition} \label{prop: SMB tensor}
Let $V$, $W$ and $Z$ be operator spaces. Then:
\begin{enumerate}
\item[(a)] The inclusion $\mathcal{SMB}(V\times W,Z)\hookrightarrow \mathcal{CB}((V\overset{h}\otimes W)\cap (W\overset{h}\otimes V), Z)$ is a complete contraction.

\item[(b)] If $Z=\mathcal L(H)$ there is a complete isomorphism $$\mathcal{SMB}(V\times W,\mathcal L(H))\cong \mathcal{CB}((V\overset{h}\otimes W)\cap (W\overset{h}\otimes V), \mathcal L(H)).$$
    \end{enumerate}
\end{proposition}
\begin{proof} (a) Composing the restriction with the usual identification we naturally have the following complete contractions:
$$
\mathcal{MB}(V\times W,Z)\hookrightarrow \mathcal{CB}((V\overset{h}\otimes W)\cap (W\overset{h}\otimes V), Z)\quad\textrm{and}\quad
\ ^t\mathcal{MB}(W\times V,Z)\hookrightarrow \mathcal{CB}((V\overset{h}\otimes W)\cap (W\overset{h}\otimes V), Z).
$$
Thus, the classical interpolation property (see \cite[Proposition 2.1]{Pisier AMS}) gives that the mapping
$$
\mathcal{SMB}(V\times W,Z)=\mathcal{MB}(V\times W,Z) + \ ^t\mathcal{MB}(W\times V,Z) \hookrightarrow \mathcal{CB}((V\overset{h}\otimes W)\cap (W\overset{h}\otimes V), Z)
$$ is also a complete contraction.

(b) In the case $Z=\mathcal L(H)$, let us see that the injective mapping of (a) is actually a surjective complete isomorphism. For that, consider $L_{\phi}\in M_n\Big(\mathcal{CB} ((V\overset{h}\otimes W)\cap (W\overset{h}\otimes V), \mathcal L(H))\Big)$. We have to prove that the bilinear associate $\phi$ belongs to $M_n\left(\mathcal{SMB}(V\times W,\mathcal L(H))\right)$ with $\|\phi\|\le 2\|L_{\phi}\|$.

Since $L_{\phi}\in M_n\Big(\mathcal{CB} ((V\overset{h}\otimes W)\cap (W\overset{h}\otimes V), \mathcal L(H))\Big)\cong
\mathcal{CB} \Big((V\overset{h}\otimes W)\cap (W\overset{h}\otimes V), \mathcal L(H^n)\Big)
$ and $(V\overset{h}\otimes W)\cap (W\overset{h}\otimes V)$ is completely isometrically contained in $(V\overset{h}\otimes W)\oplus_{\infty} (W\overset{h}\otimes V)$ there is an extension $L_{\widetilde\phi}\in
\mathcal{CB} \Big((V\overset{h}\otimes W)\oplus_{\infty} (W\overset{h}\otimes V), \mathcal L(H^n)\Big)
$ with the same completely bounded norm.
Then, we should have that the bilinear associated to $L_{\widetilde\phi}$ is written as $\phi_1 + \phi_2$ with $\|\phi_1\|_{\mathcal{MB}(V\times W, \mathcal L(H^n))}\le \|L_{\widetilde\phi}\|$ and $\|\phi_2^t\|_{\mathcal{MB}(W\times V, \mathcal L(H^n))}\le \|L_{\widetilde\phi}\|$. Hence,
$$
\|\phi_1\|_{\mathcal{MB}(V\times W, \mathcal L(H^n))} + \|\phi_2^t\|_{\mathcal{MB}(W\times V, \mathcal L(H^n))} \le 2\|L_\phi\|.
$$ Now, the usual identification $\mathcal{MB}(V\times W, \mathcal L(H^n))\cong M_n\left(\mathcal{MB}(V\times W, \mathcal L(H))\right)$ yields:
\begin{eqnarray*}
\|\phi\|_{M_n\left(\mathcal{SMB} (V\times W, \mathcal L(H))\right)} & \le & \|\phi_1\|_{M_n\left(\mathcal{MB}(V\times W, \mathcal L(H))\right)} + \|\phi_2^t\|_{M_n\left(\mathcal{MB}(W\times V, \mathcal L(H))\right)}\\ &\le & 2\|L_\phi\|_{M_n\Big(\mathcal{CB} ((V\overset{h}\otimes W)\cap (W\overset{h}\otimes V), \mathcal L(H))\Big)}.
\end{eqnarray*}
\end{proof}

  The case of scalar valued mappings is of special interest and was extensively studied in the literature   in relation with
   the so called {\sl Non-commutative Grothendieck's Theorem}. In the next section there is a  briefly exposition  of this.

  We thank the referee for suggesting  us to study  the  symmetrized multiplicatively bounded mappings and for his/her very valuable comments.

\section{Proof of Theorem \ref{thm: inclusion relations} and Examples}\label{sect: examples}

Now we study  the relationships between the bilinear ideals: we prove the inclusion relations that always hold, and  provide  examples that  distinguish them  when they are different.

\begin{proof}[Proof of Theorem \ref{thm: inclusion relations}(a)] \renewcommand{\qedsymbol}{}
It is clear, by definition, that every completely nuclear bilinear mapping is completely integral. Also,
the fact that $\|\cdot\|_\vee$ is smaller than $\|\cdot\|_h$ implies that
$$
\mathcal L_{\mathcal I}(V\overset\vee\otimes W,X)\subset \mathcal{CB}(V\overset\vee\otimes W,X)\subset \mathcal{CB}(V\overset h\otimes W,X).
$$
Moreover, since $\mathcal{I}(V\times W, X)\cong\mathcal L_{\mathcal I}(V\overset\vee\otimes W,X)$ and $\mathcal{MB}(V\times W, X)\cong \mathcal{CB}(V\overset h\otimes W,X)$, we obtain that $\mathcal{I}(V\times W, X)\subset \mathcal{MB}(V\times W, X)$.
From the very definition of $\mathcal{SMB}$, the relation  $\mathcal{MB}(V\times W, X)\subset  \mathcal{SMB}(V\times  W,X)$ always holds.
\end{proof}

All these inclusions are strict as we can see in the following  examples.

Recall that in the Banach space setting, a classical example of an integral non-nuclear bilinear mapping is $\phi:\ell_1\times\ell_1\to \mathbb{C}$ given by $\phi(x,y)=\sum_n x_ny_n$. For operator spaces, a similar example works.

\begin{example}\label{A completely integral bilinear form which is not completely nuclear}  \textit{A completely integral bilinear form which is not completely nuclear.}

Let us consider the operator space $\tau(\ell_2)$  of trace class operators from $\ell_2$ to $\ell_2$. Naturally, each element $x\in\tau(\ell_2)$ is identified with an infinite matrix $(x_{s,t})$.

We define a bilinear map $\phi:\tau(\ell_2)\times \tau(\ell_2)\to \mathbb{C}$ by
$$
\phi(x,y)=\sum_s x_{s,s}\cdot y_{s,s}
$$
The bilinear map $\phi$ is jointly completely bounded  but not completely nuclear. Indeed, by Proposition \ref{prop: nucleares bilineales=nucleares lineales into dual}, if $\phi$ is completely nuclear so is $L_\phi:\tau(\ell_2)\to \mathcal{L}(\ell_2)$ given by
$$
L_\phi(x)=\left(
            \begin{array}{ccccc}
              x_{1,1} & 0 & \cdots & \cdots & \cdots \\
              0 & x_{2,2} & 0& \cdots & \cdots \\
              0 & 0 & x_{3,3} & 0 & \cdots \\
              0 & \cdots  & 0  & x_{4,4} & \cdots \\
              \vdots  & \cdots  & \cdots  & \cdots  & \cdots  \\
            \end{array}
          \right).
$$
$L_\phi$ could not be completely nuclear because it is not compact \cite[Proposition 12.2.1]{ER-libro}.

Now we want to see that $\phi$ is completely integral. Invoking Lemma \ref{convergencia nuclear-integral}, we want to estimate the completely nuclear norms of the mappings  $\phi^m:\tau(\ell_2)\times \tau(\ell_2)\to \mathbb{C}$ given by
$$
\phi^m(x,y)=\sum_{s=1}^m x_{s,s}\cdot y_{s,s}.
$$

For each $s\in\mathbb{N}$, let us denote by $\varepsilon_{ss}$ the element in $\mathcal{L}(\ell_2)$ represented by the matrix with a number 1 in position $(s,s)$ and numbers 0 in all the other places. Recall that $\mathcal{N}(\tau(\ell_2)\times \tau(\ell_2))\cong \mathcal{L}(\ell_2)\widehat\otimes \mathcal{L}(\ell_2)/\ker\Psi$, where $\Psi:\mathcal{L}(\ell_2)\widehat\otimes \mathcal{L}(\ell_2)\to \mathcal{JCB}(\tau(\ell_2)\times \tau(\ell_2))$ is the canonical mapping defined in Section \ref{sect: compl nuclear compl integral}. Since
$$
\phi^m=\Psi\left(\sum_{s=1}^m \varepsilon_{ss}\otimes \varepsilon_{ss}\right),
$$ we have $\|\phi^m\|_{\mathcal{N}}\le \|\sum_{s=1}^m \varepsilon_{ss}\otimes \varepsilon_{ss}\|_\wedge$. In order to compute this  norm, consider  the following usual way of expressing it:
\begin{equation}\label{escritura-promedio}
\sum_{s=1}^m \varepsilon_{ss}\otimes \varepsilon_{ss}=\frac{1}{2^m}\sum_{\delta\in\{-1,1\}^m}\left(\sum_{s=1}^m \delta_s \varepsilon_{ss}\right)\otimes \left(\sum_{s=1}^m \delta_s \varepsilon_{ss}\right).
\end{equation}
It is easy to prove that for vectors $v_1,\dots,v_p$ in any operator space $V$ we have the following representation:
$$
\sum_{j=1}^p v_j\otimes v_j= \alpha\cdot \left((v_1\oplus\cdots\oplus v_p)\otimes (v_1\oplus\cdots\oplus v_p)\right) \cdot \beta,
$$ where $\alpha\in M_{1\times p^2}$, $\beta\in M_{p^2\times 1}$ and both $\alpha$ and $\beta$ have  `1' in $p$ of the places and  `0'  in the others.
Applying this representation to the expression (\ref{escritura-promedio}), we obtain
$$
\left\|\sum_{s=1}^m \varepsilon_{ss}\otimes \varepsilon_{ss}\right\|_\wedge\le \frac{1}{2^m}\|\alpha\|\cdot \max_{\delta\in\{-1,1\}^m}\left\|\sum_{s=1}^m \delta_s \varepsilon_{ss}\right\|_{\mathcal{L}(\ell_2)}^2\cdot\|\beta\|,
$$ where $\alpha\in M_{1\times 2^{2m}}$, $\beta\in M_{2^{2m}\times 1}$ and both $\alpha$ and $\beta$ have  `1' in $2^m$ of the places and  `0'  in the others. Since $\|\alpha\|=\|\beta\|= 2^{m/2}$ and $\|\sum_{s=1}^m \delta_s \varepsilon_{ss}\|_{\mathcal{L}(\ell_2)}=\max_s|\delta_s|=1$, we derive $\|\sum_{s=1}^m \varepsilon_{ss}\otimes \varepsilon_{ss}\|_\wedge\ \le 1$. Hence, $\|\phi^m\|_{\mathcal{N}}\le 1$ (in fact, it is equal to 1) and  by Lemma \ref{convergencia nuclear-integral}, $\phi$ is completely integral with $\|\phi\|_{\mathcal{I}}=1$.

\end{example}

\begin{example}\label{mb no i}
 \textit{A multiplicatively bounded bilinear mapping which is not completely integral / A symmetrized multiplicatively bounded bilinear mapping which is not multiplicatively bounded.}

Let $H$ be a Hilbert space and denote by $H_c$ the column space associated to $H$. An example of non commutativity of Haagerup tensor product is given through   the canonical complete isometries (see, for instance \cite[Propositions 9.3.1, 9.3.2 and 9.3.4]{ER-libro}):
$$
H_c\overset h\otimes(H_c)^*\cong H_c\overset\vee\otimes(H_c)^*\cong\mathcal{K}(H)\qquad \textrm{ and }\qquad (H_c)^*\overset h\otimes H_c\cong (H_c)^*\widehat\otimes H_c\cong\tau(H).
$$
A close look to these mappings allows us to state that the application
\begin{eqnarray*}
H_c\otimes(H_c)^* &\to & (H_c)^* \overset h\otimes H_c\\
v\otimes w &\mapsto & w\otimes v
\end{eqnarray*}
could not be extended  as a completely bounded mapping defined on  $ H_c\overset h\otimes (H_c)^*$. Consider
$$
\begin{array}{ccccccc}
\phi: (H_c)^*\times H_c&\to & (H_c)^* \overset h\otimes H_c &\qquad \textrm{and}\qquad &\phi^t: H_c\times(H_c)^* &\to &(H_c)^* \overset h\otimes H_c \\
(w,v) &\mapsto & w\otimes v &\qquad\ \ \ \ \qquad  &(v,w) &\mapsto & w\otimes v.
\end{array}
$$ It turns out that $\phi$ is multiplicatively bounded while $\phi^t$ is not. Hence, $\phi$ could not be completely integral (because the ideal of completely integral bilinear mappings is symmetric). Therefore, $\phi$ is multiplicatively bounded but not completely integral and $\phi^t$ is symmetrized multiplicatively bounded but not multiplicatively bounded.
\end{example}
 We also see in \cite[Example 3.6]{HaaMu}, or in Example \ref{jcb no ext} below, that  the bilinear ideals  $\mathcal{SMB}$ and $\mathcal{JCB}$ do not coincide.

\begin{proof}[Proof of Theorem \ref{thm: inclusion relations}(b)] \renewcommand{\qedsymbol}{}
The ideal of completely extendible bilinear mappings cannot be placed as a link  in the chain of inclusions in Theorem \ref{thm: inclusion relations} (a): It contains the ideal  of
 completely integral bilinear operators (see arguments below), but it has not a relation
  with the ideal of  multiplicatively bounded bilinear mappings  holding for every operator space.   Examples \ref{ext no mb} and \ref{ex: mb no e}  prove this.
   We will see, though, that  in the particularly relevant cases  when the   range is  $\mathbb{C}$ or $\mathcal{L}(H)$ there are relations between them.

In the Banach space setting, Grothendieck-integral bilinear mappings are always extendible \cite[Proposition 7]{CarLas}.
 Let us see that an analogous contention holds in the operator space framework. Pisier (personal communication) made us realize that completely integral linear mappings being  completely 2-summing are hence completely extendible \cite[Proposition 6.1]{Pisier-Asterisque}. This linear result allows us to derive the bilinear one.

 Indeed, from Theorem \ref{thm: main for integrals}, we know $\mathcal{I}(V\times W, X)\cong \mathcal{L_{\mathcal{I}}}(V\overset\vee\otimes W,X)$. Now, the previous linear inclusion gives us  $\mathcal{L_{\mathcal{I}}}(V\overset\vee\otimes W,X)\subset \mathcal{L_{\mathcal{E}}}(V\overset\vee\otimes W,X)$. Also, since the operator space tensor norm $\|\cdot\|_{\vee}$ is smaller than $\eta$,  and $ \mathcal{L}_{\mathcal{E}}$ is an ideal, we have $\mathcal{L_{\mathcal{E}}}(V\overset\vee\otimes W,X)\subset \mathcal{L_{\mathcal{E}}}(V\overset\eta\otimes W,X)$. Now, the conclusion follows once we see that given any $\varphi\in \mathcal{L_{\mathcal{E}}}(V\overset\eta\otimes W,X)$, its  associated bilinear mapping $\phi:V\times W\to X$ belongs to $\mathcal E(V\times W,X)$.

  The extendibility of $\varphi$ along with the inclusion $V\overset\eta\otimes W\hookrightarrow \mathcal L(H_V)\widehat\otimes\mathcal L(H_W)$ produce that, for any $\varepsilon >0$ there exists a completely bounded
linear mapping $\varphi_0:\mathcal L(H_V)\widehat\otimes\mathcal L(H_W)\to X$ that extends $\phi$ with
$$
\|\varphi\|_{\mathcal{L_{\mathcal{E}}}}\le \|\varphi_0\|_{cb}\le \|\varphi\|_{\mathcal{L_{\mathcal{E}}}}\ +\varepsilon.
$$
It is clear now that the bilinear map  associated to $\varphi_0$,  $\phi_0:\mathcal L(H_V)\times\mathcal L(H_W)\to X$, is an extension  of $\phi$ that satisfies
$$
\|\phi\|_{\mathcal{E}}\le \|\phi_0\|_{jcb}=\|\varphi_0\|_{cb}\le \|\varphi\|_{\mathcal{L_{\mathcal{E}}}}\ +\varepsilon.
$$
Hence, $\phi$ is completely extendible with $\|\phi\|_{\mathcal E}\le \|\varphi\|_{\mathcal{L_{\mathcal{E}}}}$.

\bigskip

Therefore, (b) in Theorem \ref{thm: inclusion relations} is proved:
$
\mathcal{I}(V\times W, X)\subset \mathcal{E}(V\times W, X)\subset \mathcal{JCB}(V\times W, X).
$

 Examples \ref{ext no mb} and \ref{ex: mb no e} below, will show     that both inclusions could be strict.
\end{proof}

 It is known \cite[page 45]{Wi} that multiplicatively bounded bilinear mappings with range $\mathcal L(H)$ are completely extendible. This can also be seen as a consequence of Arvenson-Wittstock
 extension theorem for completely bounded mappings (Remark \ref{extensiones}) along with the fact that the Haagerup tensor norm preserves complete isometries.  Moreover, the inclusion $\mathcal{MB}(V\times W, \mathcal L(H))\subset \mathcal{E}(V\times W, \mathcal L(H))$  is a complete contraction. Since $\mathcal{E}$ is a symmetric ideal, appealing once more to \cite[Proposition 2.1]{Pisier AMS}
  we derive the complete contractive inclusion
  $$
\mathcal{SMB}(V\times W, \mathcal L(H))\subset \mathcal{E}(V\times W, \mathcal L(H)),
$$
 which proves $(c)$ in Theorem \ref{thm: inclusion relations}.

We do not know whether this last inclusion is strict. Actually, for scalar-valued bilinear mappings we do know that the equality  isomorphically holds. This is a consequence of Grothendieck's Theorem for C$^*$-algebras. In \cite{Pisier-GTSurvey} one may find a broad exposition on the topic. For the moment let us recall just some relevant results in a terminology according to our presentation. Pisier and Shlyakhtenko \cite{PS} obtain the result for exact operator spaces (and also for $C^*$-algebras satisfying some conditions). In {\cite[Theorem 0.4]{PS}} they prove:

\begin{theorem HaaMu PS}[Pisier-Shlyakhtenko]
 If $V$ and $W$ are exact operator spaces, then the following isomorphism holds:
 $$\mathcal{SMB}(V\times W)= \mathcal{JCB}(V\times W).$$
\end{theorem HaaMu PS}

Haagerup and Musat \cite{HaaMu} prove the theorem for general $C^*$-algebras. Combining \cite[Theorem 1.1]{HaaMu} with \cite[Lemma 3.1]{HaaMu} (which relies on Pisier and Shlyakhtenko's result) produces:

 \begin{theorem HaaMu PS}[Haagerup-Musat]
 If $A$ and $B$ are $C^*$-algebras, then the following isomorphism holds: $$\mathcal{SMB}(A\times B)= \mathcal{JCB}(A\times B).$$
\end{theorem HaaMu PS}

 As a consequence, for any operator spaces $V$ and $W$ the following (Banach space) isomorphism holds:
 $$
 \mathcal{SMB}(V\times W)=\mathcal{E}(V\times W).
 $$
 Indeed, let $\phi\in \mathcal{E}(V\times W)$. For  $V\to\mathcal L(H_V)$ and $ W\to\mathcal L(H_W)$  complete isometries and $\varepsilon >0$, let
 $\psi: \mathcal L(H_V)\times \mathcal L(H_W)\to \mathbb{C}$ be a jointly completely bounded  extension of $\phi$ with $\|\psi\|_{jcb}\le \|\phi\|_{\mathcal E} + \varepsilon$.    By Haagerup-Musat's Theorem (for $A=\mathcal L(H_V)$ and $B=\mathcal L(H_W)$), $\psi$ can be decomposed as $\psi=\psi_1+\psi_2$, with
 $\psi_1\in \mathcal{MB}(\mathcal L(H_V)\times \mathcal L(H_W))$, $\psi_2^t\in  \mathcal{MB}(\mathcal L(H_W)\times \mathcal L(H_V))$ and $\|\psi_1\|_{mb} + \|\psi_2^t\|_{mb}\le K \|\psi\|_{jcb}$.  Restricting the domains of $\psi_1$ and $\psi_2$ to
 $V\times W$, we complete the proof.

A predual version of the last expression reads as
$$
V\overset\eta\otimes W = (V\overset h\otimes W)\cap (W \overset h\otimes V)\quad\textrm{isomorphically.}
$$

It is worth noticing that Oikhberg and Pisier in \cite{OikPis} proved that the sum of these Haagerup tensor products $(V\overset h\otimes W)+ (W \overset h\otimes V)$ is completely isometric to the ``maximal'' tensor product $V\overset \mu\otimes W$ which was introduced and studied in that article.

Let us now show that the other two inclusions of Theorem \ref{thm: inclusion relations} $(c)$ are strict. We have already distinguished the space of multiplicatively bounded bilinear forms from its symmetrized relative.   These spaces may be different even when  the range is $\mathcal L(H)$. To construct an example, first we need  an easy observation:

\begin{remark}\label{isom-mb}
Let $\phi:V\times W\to X$ be a jointly completely bounded bilinear mapping and $j:X\to Y$ be a complete isometry. Then, $\phi$ is multiplicatively bounded if and only if $j\circ \phi$ is multiplicatively bounded.

\end{remark}

Indeed, for any $v\in M_n(V)$ and $w\in M_n(W)$, since $(j\circ \phi)_{(n)}(v,w)=j_n\left(\phi_{(n)}(v,w)\right)$ we have
$$
\left\|(j\circ \phi)_{(n)}(v,w)\right\|=\left\|j_n\left(\phi_{(n)}(v,w)\right)\right\|=\left\| \phi_{(n)}(v,w)\right\|.
$$
Thus, $\|j\circ\phi\|_{mb}=\|\phi\|_{mb}$.

\begin{example} \label{ext no mb}
\textit{A symmetrized multiplicatively bounded bilinear mapping with range $\mathcal L(H)$, which is not multiplicatively bounded / A completely extendible bilinear mapping which is not completely integral.}

We recover  the mappings $\phi$ and $\phi^t$ of Example \ref{mb no i}. Denoting by $V=(H_c)^* \overset h\otimes H_c$, we consider  $\Omega_V:V\to \mathcal{L}(H_V)$ the usual completely isometric inclusion. Now, let $\psi=\Omega_V\circ\phi:(H_c)^*\times H_c\to\mathcal L(H_V)$. The previous remark and the fact that $\phi^t$ is not multiplicatively bounded, imply that $\psi^t=\Omega_V\circ \phi^t:H_c\times (H_c)^* \to\mathcal L(H_V)$  neither is multiplicatively bounded.

On the other hand,  $\phi\in \mathcal{MB}((H_c)^*\times H_c,(H_c)^* \overset h\otimes H_c)$ and so $\psi\in \mathcal{MB}((H_c)^*\times H_c,\mathcal L(H_V))$. Hence,
 $\psi^t\in \mathcal{SMB}((H_c)^*\times H_c,\mathcal L(H_V))$.
\end{example}

\begin{example}\label{jcb no ext}  \textit{A jointly completely bounded bilinear mapping (with range $\mathbb C$) which is not extendible (and hence not symmetrized multiplicatively bounded).}

 Consider a non-complemented copy of $\ell_2$ in $\mathcal L(H)$,  and let $V$ be the operator space determined  by $\ell_2$ with  the matrix structure inherited from $\mathcal L(H)$.
 Let
$$
\begin{array}{ccc}
\phi: V\times V^* &\to & \mathbb{C}  \\
((a_i)_i, (b_i)_i) &\mapsto & \sum_{i=1}^{\infty}a_ib_i.
\end{array}
$$
$\phi$ is jointly completely bounded  but there is  not a jointly completely bounded  extension of $\phi$ defined on    $\mathcal L(H)\times V^*$, since this extension would give rise to a bounded projection
on $\mathcal L(H)$ onto that copy of $\ell_2$.
\end{example}

Now we  prove that the inclusion of the space of multiplicatively bounded bilinear mappings (and hence symmetrized multiplicatively bounded) into the space of completely extendible bilinear mappings is not longer true when the range space is an arbitrary operator space.

For that, it is convenient  to introduce  the concept of completely extendible linear mapping. We say that a mapping    $\varphi\in \mathcal{CB}(V,Z)$ is \textbf{completely extendible} if for any operator space $X$ such that $V\subset X$, there exists a completely bounded extension $\overline{\varphi}:X\to Z$ of $\varphi$. The set of completely extendible linear mappings from $V$ to $Z$ is denoted by $\mathcal{L}_{\mathcal{E}}(V,Z)$.

Following the same steps as in the proofs of Proposition \ref{prop:extendible a L(H)} it is obtained that $\varphi\in \mathcal{CB}(V,Z)$ is completely extendible if and only if it can be extended to $\mathcal L(H_V)$ and that $\mathcal{L}_{\mathcal{E}}(V,Z)$ is an operator space with the norm given by
$$
\|\varphi\|_{\mathcal L_{\mathcal E}}= \inf \{\|\varphi_0\|_{cb}:\ \varphi_0 \textrm{ extension of }\varphi \textrm{ to } M_n\left(\mathcal{CB} (\mathcal L(H_V), Z)\right)\},
$$ for every $\varphi\in M_n\left(\mathcal{L}_{\mathcal E}(V, Z)\right)$.

As in Proposition \ref{prop:extendible es ideal} it is also obtained that $\mathcal L_{\mathcal E}$ is a (linear) mapping ideal.

\begin{example}\label{ex: mb no e} \textit{A  multiplicatively bounded bilinear mapping  which is not extendible.}

Let $V$ be the operator space of Example \ref{jcb no ext}. The canonical mapping   $V \overset h\otimes \mathbb C\rightarrow V $ is a complete isometry. Hence,  its associated bilinear map  $\phi: V\times \mathbb C\rightarrow V$  is  multiplicatively bounded. However, since $id: V \rightarrow V $   is not extendible,  $\phi$  neither is so.

\end{example}

\subsection*{Acknowledgements} The first author wishes to thank the Centro de Investigaci\'{o}n en Matem\'{a}ticas (Guanajuato) for its kind hospitality during the months of January and February 2012, when this work was initiated.

\end{document}